\documentclass[reqno,10pt]{amsart}

\usepackage{amsmath, amsthm}
\usepackage{pdfsync}
\usepackage{parskip}
\usepackage{hyperref}
\usepackage{mathtools}
\usepackage{mdframed}
\usepackage{tikz}
\numberwithin{equation}{section}
\newtheorem{theorem}{Theorem}[section]
\newtheorem{prop}[theorem]{Proposition}

\newtheorem{lem}[theorem]{Lemma}
\newtheorem{cor}[theorem]{Corollary}

\theoremstyle{remark}
\newtheorem{remark}[theorem]{Remark}
\newtheorem{rmk}[theorem]{Remarks}

\def\R{\mathbb{R}}
\def\H{\mathbb{H}}
\def\L{\mathcal{L}}
\def\Lis{\mathcal{L}{\rm{is}}}
\def\id{{\rm{id}}}
\def\ev{\mathsf{C}}
\def\supp{{\rm{supp}}}
\def\Ric{{\rm{Rc}}}

\def\uf{{\rm{unif}}}
\def\cp{{\rm{cp}}}
\def\kp{\kappa_{\sf{p}}}
\def\M{\mathsf{M}}
\def\p{\mathsf{p}}
\def\x{x_{\mathsf{p}}}
\def\B{\mathbb{B}(0,r)}
\def\U{\mathsf{O}}
\def\Uk{\mathsf{O}_{\kappa}}
\def\Uki{\mathsf{O}_{\iota}}
\def\Ukp{\mathsf{O}_{\kp}}
\def\Q{\mathbb{B}^{m}}
\def\Qk{\mathbb{B}^{m}_{\kappa}}
\def\Qi{\mathbb{B}^{m}}

\def\pk{\pi_\kappa}

\def\F{\mathfrak{F}}
\def\N{\mathbb{N}}
\def\tm{\theta^{\ast}_{\mu}}
\def\ttm{{\Theta}^{\ast}_{\mu}}
\def\ttmi{{\Theta}^{\mu}_{\ast}}
\def\tl{\theta^{\ast}_{\lambda,\mu}}
\def\ttl{{\Theta}^{\ast}_{\lambda,\mu}}
\def\tu{{T}_{\mu}}
\def\tui{{T}^{-1}_{\mu}}
\def\ttu{\tilde{T}_{\mu}}
\def\ttui{\tilde{T}^{-1}_{\mu}}
\def\kb{\varphi^{\ast}_{\iota}}
\def\kf{\psi^{\ast}_{\iota}}
\def\kbk{\varphi^{\ast}_{\kappa}}
\def\kfk{\psi^{\ast}_{\kappa}}
\def\kft{\psi^{\ast}_{\tilde{\kappa}}}
\def\kbp{\varphi^{\ast}_{\kp}}
\def\kfp{\psi^{\ast}_{\kp}}
\def\ut{{u}_{\lambda,\mu}}
\def\rh{{\varrho}^{\ast}_{\lambda}}
\def\tu{{T}_{\mu}}
\def\ez{\mathbb{E}_{0}}
\def\ef{\mathbb{E}_{1}}
\def\sez{\mathbb{SE}_{0}}
\def\sef{\mathbb{SE}_{1}}

\def\Ub{\mathbb{U}}

\def\K{\mathbb{K}}

\def\Re{\mathcal{R}}
\def\Rek{\mathcal{R}_{\kappa}}
\def\Rc{\mathcal{R}^c}
\def\Rck{\mathcal{R}^c_{\kappa}}

\def\Dfi{\rm{Diff}^{\hspace{.1em}\infty}}
\def\im{\rm{im}}
\def\Hom{{\text{Hom}}}
\def\X{\mathbb{X}}
\def\Xk{\mathbb{X}_{\kappa}}
\def\Xo{\tilde{\varsigma}}
\def\Xt{\varsigma}
\def\Am{{\mathcal{A}}_{\mu}}
\def\Al{{\mathcal{A}}_{\mu}}
\def\Af{{\mathcal{A}}_{1}(\mu)}
\def\As{{\mathcal{A}}_{2}(\mu)}

\begin{document}

\title[Parameter-Dependent Diffeomorphisms and Applications]{A Family of Parameter-Dependent Diffeomorphisms Acting on Function Spaces over a Riemannian Manifold and Applications to Geometric Flows}

\author[Y. Shao]{Yuanzhen Shao}
\address{Department of Mathematics,
         Vanderbilt University, 
         1326 Stevenson Center, 
         Nashville, TN 37240, USA}
\email{yuanzhen.shao@vanderbilt.edu}


\begin{abstract}
It is the purpose of this article to establish a technical tool to study regularity of solutions to parabolic equations on manifolds. As applications of this technique, we prove that solutions to the Ricci-DeTurck flow, the surface diffusion flow and the mean curvature flow enjoy joint analyticity in time and space, and solutions to the Ricci flow admit temporal analyticity.
\end{abstract}
\maketitle

\section{\bf Introduction}

In this paper, we establish a localized translation technique on a Riemannian manifold $\M$, which induces a family of parameter-dependent diffeomorphisms acting on tensor fields over $\M$. The background manifolds of geometric interest to us are called uniformly regular Riemannian manifolds, a concept introduced by H. Amann in \cite{Ama13, AmaAr}. These manifolds may be non-compact, or even non-complete. 

In some publications, regularity of solutions to partial differential equations is established by means of the implicit function theorem in conjunction with a translation argument, see for example \cite{Ange902,EscPruSim0302,EscSim96,Lunar95}. More precisely, one introduces parameters representing translation in space and time into the solution of some differential equation. Then we study the parameter-dependent equation satisfied by this transformed solution. The implicit function theorem yields the smooth dependence of the solution to the parameter-dependent problem upon the parameters. This regularity property is inherited by the original solution. An advantage of this technique is reflected by its power to prove analyticity of solutions to differential equations, which is not approachable through the classical method of bootstrapping. We first consider the usual translation $(t,x)\mapsto(t+\lambda,x+\mu)$ involving both spatial and time variables. However, the global nature of this transformation creates a barrier to applying it to functions over manifolds. So we desire an alternative that only shifts the variables ``locally", which empowers us to define the translation in some local chart without changing the value of the functions outside. The idea of a localized translation is first put to use in \cite{EscPruSim03} to study regularity of solutions to elliptic and parabolic equations in Euclidean spaces. The basic building block of \cite{EscPruSim03} is rescaling translations by some cut-off function so that they vanish outside a precompact neighbourhood. The results therein break down the aforementioned barrier and thus enable a local action $\Theta_{\lambda,\mu}$ with parameters $(\lambda,\mu)\in{\R}\times{\R}^m$ to be defined on a manifold $\M$, which induces a parameter-dependent diffeomorphism ${\ttl}$ acting on functions, or tensor fields, over $\M$. Throughout, ${\M}$ always denotes an $m$-dimensional uniformly regular Riemannian manifold, and for any interval $I$ containing $0$, $\dot{I}:= I \setminus\{0\}$. For any topological set $U$, $\mathring{U}$ denotes the interior of $U$. In the case that $U$ consists of only one point, we set $\mathring{U}:=U$. The main result of this paper can be stated as follows:
\begin{theorem}
\label{S1: Theorem}
Let $k\in\N\cup\{\infty,\omega\}$ with $\omega$ being the symbol for real analyticity. Suppose that $\M$ is a $C^k$-uniformly regular Riemannian manifold, and $u\in{BC(I\times{\M})}$. Then $u\in{C^{k}(\mathring{I}\times\mathring{\M})}$ iff for any $(t_0,\p) \in\mathring{I}\times\mathring{\M}$, there exists $r=r(t_0,\p)>0$ and a family of corresponding parameter-dependent diffeomorphisms $\{\ttl:(\lambda,\mu)\in\B\}$ such that 
\begin{align*}
[(\lambda,\mu)\mapsto{\ttl}u]\in{C^{k}({\B},BC(I\times{\M}))}.
\end{align*}
\end{theorem}
Suppose that $u$ is the solution to some differential equation. As we will see in the following, maximal regularity theory established in \cite{ShaoSim13} along with an implicit function theorem argument allows us to obtain the smooth dependence of the solution ${\ttl}u$ to some parameter-dependent problem upon the translation parameters. Theorem~\ref{S1: Theorem} then avails us in acquiring regularity of the solution $u$ to the original differential equation. 

There are two types of geometric evolution equations stimulating my interest in developing the aforementioned technique. First of them is the evolution of metrics via certain differential equations. One of the most famous and representative examples among them is the Ricci flow: 
\begin{align}
\label{S1: Ricci}
\partial_t g=-2\Ric(g), \hspace{1em} g(0)=g_0,
\end{align}
where $\Ric(g)$ is the Ricci tensor with respect to the evolving metric $g$ and $g_0$ is the initial metric. The study of this equation is initiated by R. Hamilton in his groundbreaking paper \cite{Ham82}. It serves as the primary tool in G. Perelman's solution \cite{Perel59, Perel09} to the Poincar\'e conjecture. Many other authors also contribute to this subject from different perspectives, see \cite{Cai10, Chow04, MorTian07}. R.~Hamilton in \cite{Ham82} proves short-time existence of a smooth solution to the evolution equation \eqref{S1: Ricci} with smooth initial data $g_0$. But his formulation of the problem loses parabolicity, and the proof is based on the Nash-Moser theorem. Shortly afterwards, D.~DeTurck \cite{DeTur83} modifies this equation to be parabolic by removing the symmetry caused by the invariance of the Ricci tensor under diffeomorphisms, which forces the original equation to be weakly parabolic. The investigation into analyticity of solutions commences with a result of S.~Bando \cite{Bando87} showing that, for $0<t<T$ with some positive $T$, the manifold $(\M,g(t))$ is real analytic in normal coordinates with respect to $g(t)$. On non-compact manifolds, B.~Kotschwar \cite{Kots13} proves a local version of Bando's theorem without a global bound of the Riemannian curvature tensor. A temporal analyticity result is stated in \cite{KotsArx}. The author also shows a local time-space analyticity result in normal coordinates therein. In Section 4, we will see that solutions to the Ricci-DeTurck flow
\begin{align}
\partial_t g=-2\Ric(g)+\L_{W_g}g, \hspace{1em} g(0)=g_0,
\end{align}
enjoy joint analyticity in time and space for any initial metric in the class $C^{2+\alpha}$ with respect to a fixed atlas. Here $\L_{W_g}g$ denotes the Lie derivative of $g$ with respect to a vector field $W_g$, see Section~4 for a precise definition. Based on this result, we will present an alternative proof for the time-analyticity of solutions to the Ricci flow, which is shorter than the one for \cite[Theorem~1]{KotsArx}. We shall point out that, with an arbitrary $C^{\omega}$-atlas, spatial analyticity of solutions to \eqref{S1: Ricci} with smooth initial metric in general cannot be true. Indeed, if $g$ is an analytic solution to \eqref{S1: Ricci} with $g_0$ smooth, then for every smooth, but not analytic, diffeomorphism $\phi:\M\rightarrow\M$, $\phi^{\ast}{g}$ solves \eqref{S1: Ricci} with smooth initial metric $\phi^{\ast}g_0$. However, in general, $\phi^{\ast}{g}$ is not analytic. This observation suggests that in order to prove the spatial or joint analyticity of the solution ${g}$ with arbitrary fixed atlas, we should expect the initial metric $g_0$ to be analytic as well.
\smallskip\\
The second type of geometric evolution equations deals with the deformation of manifolds $\Gamma(t)$ driven by by their curvatures, including the surface diffusion flow
\begin{align}
\label{S1: SDF}
V(t)=-\Delta_{\Gamma(t)}H_{\Gamma(t)},\hspace{1em} \Gamma(0)=\Gamma_0,
\end{align}
and the averaged mean curvature flow
\begin{align}
\label{S1: MCF}
V(t)=H_{\Gamma(t)}-h_{\Gamma(t)},\hspace{1em} \Gamma(0)=\Gamma_0,
\end{align}
where $V(t)$ and $H_{\Gamma(t)}$ denote the normal velocity and the mean curvature of $\Gamma(t)$, respectively. $h_{\Gamma(t)}$ stands for the average of $H_{\Gamma(t)}$ on $\Gamma(t)$. Meanwhile, $\Delta_{\Gamma(t)}$ is the Laplace-Beltrami operator on ${\Gamma(t)}$. Existence and uniqueness of a local smooth solution to equation~\ref{S1: SDF} has been established by J.~Escher, U.~Mayer and G.~Simonett in \cite{EscMaySim98}. Later, it is proven by J.~Escher and P.B.~Mucha \cite{EscMuc10} that this result admits initial surfaces in Besov spaces. Existence results for initial surfaces with lower regularity are obtained in \cite{Koch12} for graph-like hypersurfaces. Results concerning lifespans of solutions to \eqref{S1: SDF} can be found in \cite{EscMaySim98, Wheel11, Wheel12}. The reader may refer to \cite{LecSim13} for a more detailed historical account of this problem. 
\smallskip\\
The averaged mean curvature flow is introduced by M. Gage \cite{Gage86} and G. Huisken \cite{Huis87}. In \cite{Huis87}, the author proves global existence of smooth solutions for a smooth, uniformly convex initial surfaces. This constraint on the initial data is later loosed by J.~Escher and G.~Simonett \cite{EscSim98} to admit any initial surface in $C^{1+\alpha}$. The prototype of \eqref{S1: MCF}, the (unaveraged) mean curvature flow, which is obtained by removing the term $h_{\Gamma(t)}$ in \eqref{S1: MCF}, is first studied in the pioneering work of K. Brakke \cite{Brak78}. Like equation \eqref{S1: MCF}, solutions to the mean curvature flow smooth the initial data immediately. A core theme in the study of the mean curvature flow is to investigate the development and structure of singularities, see e.g. \cite{Huis99a, Huis99b, Huis09, White00, White03}. Similar to the Ricci-DeTurck flow, it will be shown in Section 5 and 6 that solutions to the three above-mentioned equations are indeed real analytic jointly in time and space. 

Section 2 is of preparatory character, wherein we introduce some basic concepts and properties of function spaces and tensor fields on a uniformly regular Riemannian manifold $\M$. They serve as the stepstone to the theory of differential equations on $\M$, see \cite{Ama13, AmaAr,ShaoSim13} for related work.
\smallskip\\
Section 3 plays a central role in this paper. Therein we introduce the aforesaid parameter-dependent diffeomorphism technique. A major obstruction of bringing in the localized translations on $\M$ lies in how to introduce parameters so that the transformed functions and differential operators depend ideally on the parameters as long as the original functions and operators are smooth enough around the ``center" of the localized translations. Thanks to the discussions in \cite{ShaoSim13} and Section~2, we can set up these properties based on the prototypical results in \cite{EscPruSim03}. 

Although in this paper our discussion is restricted to geometric evolution equations on compact manifolds, the reader should be aware that the techniques herein, as is shown in \cite{ShaoSim13} for the Yamabe flow on non-compact and non-complete manifolds, also apply to equations on general uniformly regular Riemannian manifolds. 

We mention that the method to combine the implicit function theorem and translations to study regularity of solutions to parabolic evolution equations on compact manifolds is also addressed in \cite{EscPro06} by J. Escher and G. Prokert. In their work, the authors first construct a family of global real analytic vector fields. The induced one-parameter group of diffeomorphisms plays the same role as ${\ttl}$ in this paper. It is known that the existence of such a one-parameter group is valid on compact manifolds, but is not guaranteed in general. 

{\textbf {Assumptions on manifolds:}}
We list some geometric assumptions on the Riemannian manifolds treated in this paper. They provide the basis for analyzing function spaces and tensor fields on manifolds. This work is originally addressed in \cite{Ama13, AmaAr}. 
\smallskip\\
Let $({\M},g)$ be a $C^{\infty}$-Riemannian manifold of dimension $m$ with or without boundary endowed with $g$ as its Riemannian metric such that its underlying topological space is separable and metrizable. An atlas $\mathfrak{A}:=({\Uk},\varphi_{\kappa})_{{\kappa}\in \mathfrak{K}}$ for ${\M}$ is said to be normalized if 
\begin{align*}
\varphi_{\kappa}({\Uk})=
\begin{cases}
\Q, \hspace*{1em}& \Uk\subset\mathring{\M},\\
\Q\cap\H^m, &\Uk\cap\partial\M \neq\emptyset,
\end{cases}
\end{align*}
where $\H^m$ is the closed half space $\R^{+}\times\R^{m-1}$ and $\Q$ is the unit Euclidean ball centered at the origin in ${\R}^m$. We put $\Qk:=\varphi_{\kappa}({\Uk})$ and  $\psi_{\kappa}:=\varphi_{\kappa}^{-1}$. 
\smallskip\\
The atlas $\mathfrak{A}$ is said to have \emph{finite multiplicity} if there exists $K\in{\N}$ such that any intersection of more than $K$ coordinate patches is empty. Put
\begin{align*}
\mathfrak{N}(\kappa):=\{\eta\in\mathfrak{K}:\mathsf{O}_{\eta}\cap\Uk\neq\emptyset \}.
\end{align*} 
The finite multiplicity of $\mathfrak{A}$ and the separability of $\M$ imply that $\mathfrak{K}$ is countable.
If two real-valued functions $f$ and $g$ are equivalent in the sense that $f/c\leq g\leq cf$ for some $c\geq 1$, then we write $f\sim g$.
\smallskip\\
An atlas $\mathfrak{A}$ is said to fulfil the \emph{uniformly shrinkable} condition, if it is normalized and there exists $0<r<1$ such that $\{\psi_{\kappa}(r{\Qk}):\kappa\in\mathfrak{K}\}$ is a cover for ${\M}$. 
\smallskip\\
Following H. Amann \cite{Ama13, AmaAr}, we say that a manifold $(\M,g)$ endowed with an atlas $\mathfrak{A}$ is a {\bf{$C^{\infty}$-uniformly regular Riemannian manifold}}, or simply {\bf{uniformly regular Riemannian manifold}},  if 
\begin{itemize}
\item[(R1)] $\mathfrak{A}$ is uniformly shrinkable and has finite multiplicity.
\item[(R2)] $\|\varphi_{\eta}\circ\psi_{\kappa}\|_{k,\infty}\leq{c(k)}$, $\kappa\in\mathfrak{K}$, $\eta\in\mathfrak{N}(\kappa)$, and $k\in{\N}_0$.
\item[(R3)] ${\psi_{\kappa}^{\ast}}g\sim{g_m}$, $\kappa\in\mathfrak{K}$. Here $g_m$ denotes the Euclidean metric on ${\R}^m$ and ${\psi_{\kappa}^{\ast}}g$ denotes the pull-back metric of $g$ by ${\psi_{\kappa}}$.
\item[(R4)] $\|{\psi_{\kappa}^{\ast}}g\|_{k,\infty}\leq c(k)$, $\kappa\in\mathfrak{K}$ and $k\in\N_0$.
\end{itemize}
Here $\|u\|_{k,\infty}:=\max_{|\alpha|\leq k}\|\partial^{\alpha}u\|_{\infty}$, and it is understood that a constant $c(k)$, like in (R2), depends only on $k$. ${\N}_0$ is the set of all natural numbers including $0$. An atlas $\mathfrak{A}$ satisfying (R1) and (R2) is called a \emph{uniformly regular atlas}. (R3) reads as 
\begin{align*}
|\xi|^2/c\leq{\psi_{\kappa}^{\ast}g(x)(\xi,\xi)}\leq{c|\xi|^2},\hspace{.5em} \text{ for any } x\in{\Qk},\xi\in{\R}^m \text{ and some } c\geq{1}
\end{align*}
uniformly in $\kappa$. In particular, any compact manifold is uniformly regular.
The reader may consult \cite{Ama13b} for examples of uniformly regular Riemannian manifolds.
\smallskip\\
Given any Riemannian manifold $\M$ without boundary, in virtue of a result of R.E. Greene \cite{Greene78} stating that there exists a complete Riemannian metric with bounded geometry on $\M$, we can always find a Riemannian metric $g_c$ making $(\M,g_c)$ uniformly regular, see \cite[Theorem~2']{Greene78} and \cite[Remark~1.7]{MullerAx}. However, this result is of restricted interest, since in most of the PDE problems, we are forced to work with a fixed background metric whose compatibility with the metric $g_c$ is unknown.
\smallskip\\
A uniformly regular Riemannian manifold $\M$ admits a \emph{localization system subordinate to} $\mathfrak{A}$, by which we mean a family $(\pi_{\kappa},\zeta_{\kappa})_{\kappa\in\mathfrak{K}}$ satisfying:
\begin{itemize}
\item[(L1)] ${\pk}\in\mathcal{D}({\Uk},[0,1])$ and $(\pi_{\kappa}^{2})_{\kappa\in{\mathfrak{K}}}$ is a partition of unity subordinate to $\mathfrak{A}$.
\item[(L2)] $\zeta_{\kappa}:={\kbk}\zeta$ with $\zeta\in\mathcal{D}({\Q},[0,1])$ satisfying that $\zeta|_{\supp({{\kfk}\pi_{\kappa}})}\equiv 1$, $\kappa\in\mathfrak{K}$.
\item[(L3)] $\|\psi_{\kappa}^{\ast}{\pk}\|_{k,\infty}+\|\zeta\|_{k,\infty} \leq{c(k)}$, for $\kappa\in\mathfrak{K}$, $k\in{\N}_0$.
\end{itemize}
The reader may refer to \cite[Lemma~3.2]{Ama13} for a proof. If, in addition, the atlas $\mathfrak{A}$ and the metric $g$ are real analytic, we say that $(\M,g)$ is a {\bf{$C^{\omega}$-uniformly regular Riemannian manifold}}.
\smallskip\\
Lastly, for each $k\in\N$, the concept of {\bf{$C^k$-uniformly regular Riemannian manifold}} is defined by modifying (R2), (R4) and (L1)-(L3) in an obvious way.

{\textbf{Notations:}}
Let ${\K}\in\{{\R},\mathbb{C}\}$. For any open subset $U\subseteq{\R}^m$, we abbreviate ${\F}^{s}(U,\mathbb{K})$ to ${\F}^{s}(U)$, where $s\geq 0$ and ${\F}\in \{bc,BC,W_p,H_p\}$. The precise definitions for these function spaces will be presented in Section 2. Similarly, ${\F}^{s}({\M})$ stands for the corresponding $\mathbb{K}$-valued spaces defined on the manifold ${\M}$. 
\smallskip\\
Let $\| \cdot \|_{\infty}$, $\| \cdot \|_{s,\infty}$, $\| \cdot \|_{p}$, $\| \cdot \|_{s,p}$ and 
$\|\cdot\|_{H^{s}_p}$ denote the usual norm of the Banach spaces $BC(U),BC^{s}(U)$, $L_p(U),W^{s}_p(U),H^{s}_p(U)$, respectively. Their counterparts on ${\M}$ are expressed by $\| \cdot \|_{\F}^{\M}$ with $\| \cdot \|_{\F}$ being any of the norms defined on $U$.
\smallskip\\
The notation $\mathcal{T}^{\sigma}_{\tau}{\M}$ stands for the $C^{\infty}({\M})$-module of all smooth sections of the $(\sigma,\tau)$-tensor bundle $T^{\sigma}_{\tau}{\M}:=T{\M}^{\otimes{\sigma}}\otimes{T^{\ast}{\M}^{\otimes{\tau}}}$  for $\sigma,\tau\in{\N}_0$, where $T{\M}$ and $T^{\ast}{\M}$ denote the tangent and the cotangent bundle of ${\M}$, respectively. For abbreviation, we set $\mathbb{J}^{\sigma}:=\{1,2,\ldots,m\}^{\sigma}$, and $\mathbb{J}^{\tau}$ is defined alike. Given a coordinate $\psi=\{x^1,\ldots,x^m\}$, $(i):=(i_1,\ldots,i_{\sigma})\in\mathbb{J}^{\sigma}$ and $(j):=(j_1,\ldots,j_{\tau})\in\mathbb{J}^{\tau}$, we set
\begin{align*}
\frac{\partial}{\partial{x}^{(i)}}:=\frac{\partial}{\partial{x^{i_1}}}\otimes\cdots\otimes\frac{\partial}{\partial{x^{i_{\sigma}}}}, \hspace*{.5em} dx^{(j)}:=dx^{j_1}\otimes{\cdots}\otimes{dx}^{j_{\tau}}.
\end{align*}
The local representation of $a\in \mathcal{T}^{\sigma}_{\tau}{\M}$ with respect to these coordinates is given by 
\begin{align*}
a=a^{(i)}_{(j)} \frac{\partial}{\partial{x}^{(i)}} \otimes dx^{(j)}, \hspace{1em}\text{ with } a^{(i)}_{(j)}\in C^{\infty}(\Uk).
\end{align*}
For any two Banach spaces $X,Y$, $X\doteq Y$ means that they are equal in the sense of equivalent norms.
The notation $\Lis(X,Y)$ stands for the set of all bounded linear isomorphisms from $X$ to $Y$.
\medskip

\section{\bf Function Spaces on Uniformly Regular Riemannian Manifolds}

In Section~2, we will review some prerequisites on function spaces and tensor fields over uniformly regular Riemannian manifolds. These materials are developed by H. Amann in  \cite{Ama13, AmaAr} for weighted sections of tensor fields defined on manifolds with ``singular ends" characterized by a ``singular function" $\rho\in C^{\infty}(\M,(0,\infty))$. Such manifolds are uniformly regular iff the singular datum $\rho\sim 1_{\M}$. The work therein has been employed to establish continuous maximal regularity on uniformly regular Riemannian manifolds for parabolic differential operators in \cite{ShaoSim13}. I will state the definitions and some properties of these spaces without giving proofs. The reader may refer to the aforementioned references in this paragraph for details.

Let $(\M,g)$ be a uniformly regular Riemannian manifold. The extension of the Levi-Civita connection on $\mathcal{T}^{\sigma}_{\tau}{\M}$ is denoted by $\nabla=\nabla_g$. Set $\nabla_{i}:=\nabla_{\partial_{i}}$ with $\partial_{i}=\frac{\partial}{\partial{x^i}}$. The generalized metric $g^{\tau}_{\sigma}$ on $T^{\sigma}_{\tau}{\M}$ is still written as $(\cdot|\cdot)_{g}$. Meanwhile $(\cdot|\cdot)_{g^{\ast}}$ stands for the induced contravariant metric. In addition,
\begin{align*}
|\cdot|_g:\mathcal{T}^{\sigma}_{\tau}{\M}\rightarrow{C^{\infty}}({\M}),\hspace*{.5em} a\mapsto\sqrt{(a|a)_g}
\end{align*}
is called the (vector bundle) \emph{norm} induced by $g$.
\smallskip\\
Henceforth, we assume that $V$ is a $\K$-valued tensor bundle on $\M$, i.e.,
\begin{align*}
V=V^{\sigma}_{\tau}:=\{T^{\sigma}_{\tau}\M, (\cdot|\cdot)_g\} 
\end{align*}
for some $\sigma,\tau\in\N_0$, and 
\begin{align*}
E=E^{\sigma}_{\tau}:=\{\K^{m^{\sigma}\times m^{\tau}},(\cdot|\cdot)\}, 
\end{align*}
where $(a|b):=$trace$(b^{\ast}a)$ with $b^{\ast}$ being the conjugate matrix of $b$. By setting $N=m^{\sigma+\tau}$ , we can identify $\F^s(\M,E)$ with $\F^s(\M)^N$. The notation $\Gamma(\M,V)$ stands for the set of all sections of the $(\sigma,\tau)$-tensor bundle. 
\smallskip\\
Throughout the rest of this paper, we always assume that 
\smallskip
\begin{mdframed}
\begin{itemize}
\item $({\M},g)$ is a uniformly regular Riemannian manifold.
\item $s\geq 0$, and $1 < p < \infty$.
\item $(\pi_{\kappa},\zeta_{\kappa})_{\kappa\in\mathfrak{K}}$ is a localization system subordinate to $\mathfrak{A}$.
\item $\sigma,\tau\in \N_0$, $V=V^{\sigma}_{\tau}:=\{T^{\sigma}_{\tau}\M, (\cdot|\cdot)_g\}$ and $E=E^{\sigma}_{\tau}:=\{\K^{m^{\sigma}\times m^{\tau}},(\cdot|\cdot)\}$.
\end{itemize}
\end{mdframed}
\smallskip
For $K\subset \M$, we put $\mathfrak{K}_{K}:=\{\kappa\in \mathfrak{K}: \Uk\cap K\neq\emptyset\}$. Then, given $\kappa\in\mathfrak{K}$
\begin{align*}
\Xk:=\begin{cases}
\R^m \hspace*{1em}\text{if }\kappa\in \mathfrak{K}\setminus \mathfrak{K}_{\partial\M},\\
\H^m \hspace*{1em}\text{otherwise,}
\end{cases}
\end{align*}
endowed with the Euclidean metric $g_m$. Given $a\in \Gamma(\M,V)$, we define ${\kfk}a$ in $E$ as
\begin{align*}
{\kfk}a=[a^{(i)}_{(j)}],\hspace*{1em}\text{ with }(i)\in\mathbb{J}^{\sigma}, (j)\in\mathbb{J}^{\tau}.
\end{align*}
Here $[a^{(i)}_{(j)}]$ stands for the $(m^{\sigma}\times m^{\tau})$-matrix with $a^{(i)}_{(j)}$ in the $((i),(j))$ position, and $(i)$, $(j)$ are arranged lexicographically. Next we define:
\begin{align*}
\Rck:{L}_{1,loc}({\M},V)\rightarrow{L}_{1,loc}({\Xk},E),\hspace*{.5em} u\mapsto{\psi_{\kappa}^{\ast}({\pk}u)},
\end{align*}
and
\begin{align*}
\Rek:{L}_{1,loc}({\Xk},E)\rightarrow{L}_{1,loc}({\M},V),\hspace*{.5em} v_{\kappa}\mapsto{\pk}{\kbk}v_{\kappa}.
\end{align*}
Here and in the following it is understood that a partially defined and compactly supported tensor field is automatically extended over the whole base manifold by identifying it to be zero sections outside its original domain. Moreover, let	
\begin{align*}
\Rc:{L}_{1,loc}({\M},V)\rightarrow\boldsymbol{L}_{1,loc}(\X,E),\hspace*{.5em} u\mapsto{(\Rck u)_{\kappa}},
\end{align*}
and
\begin{align*}
\Re:\boldsymbol{L}_{1,loc}(\X,E)\rightarrow{L}_{1,loc}({\M},V),\hspace*{.5em} (v_{\kappa})_{\kappa}\mapsto{\sum_{\kappa}\Rek v_{\kappa}}
\end{align*}
with $\boldsymbol{L}_{1,loc}(\X,E):=\prod_{\kappa}{L}_{1,loc}({\Xk},E)$.
Let $\mathbb{A}$ be a countable index set and $E_{\alpha}$ be Banach spaces for $\alpha\in\mathbb{A}$. Then we put $\mathbf{E}:=\prod_{\alpha}E_{\alpha}$. We endow $\mathbf{E}$ with the product topology, that is, the the coarsest topology for which all projections $pr_{\beta}:\mathbf{E}\rightarrow{E_{\beta}},(e_{\alpha})_{\alpha}\mapsto{e_{\beta}}$ are continuous. For $1\leq{q}\leq{\infty}$, we denote by $l_q(\mathbf{E})$ the linear subspace of $\mathbf{E}$ consisting of all $\boldsymbol{x}=(x_{\alpha})$ such that
\begin{align*}
\|\boldsymbol{x}\|_{l_q(\mathbf{E})}:=
\begin{cases}
(\sum_{\alpha}\|x_{\alpha}\|^q_{E_{\alpha}})^{1/q},\hspace{1em}&1\leq{q}<{\infty},\\
\sup_{\alpha}\|x_{\alpha}\|_{E_{\alpha}},& \hspace{.5em}q=\infty
\end{cases}
\end{align*}
is finite. Then $l_q(\mathbf{E})$ is a Banach space with norm $\|\cdot\|_{l_q(\mathbf{E})}$.
\smallskip\\
For ${\F}\in \{bc,BC,W_{p},H_{p} \}$, we put $\boldsymbol{\F}^s:=\prod_{\kappa}{\F}^s_{\kappa}$, where ${\F}^s_{\kappa}:={\F}^s(\Xk,E)$. 

Throughout the rest of this subsection, let $k\in{\N}_{0}$. The Sobolev space $W^{k}_p({\M},V)$ is defined as the completion of $\mathcal{D}({\M},V)$ in $L_{1,loc}({\M},V)$ with respect to the norm
\begin{align*}
\|\cdot\|^{\M}_{k,p}: u\mapsto(\sum_{i=0}^{k}\||\nabla^{i}u|_{g}\|_p^p)^{\frac{1}{p}}.
\end{align*}
Note that $W^{0}_p(\M,V)\doteq L_p(\M,V)$. The Bessel potential spaces are defined by means of interpolation
\begin{align*}
H^{s}_p({\M},V):=
\begin{cases}
[W^{k}_p({\M},V),W^{k+1}_p({\M},V)]_{s-k} \hspace{1em}&\text{ for }k<s<k+1,\\
[W^{k-1}_p({\M},V),W^{k+1}_p({\M},V)]_{1/2} \hspace{1em}&\text{ for }s=k\in\N,\\
L_p({\M},V)&\text{ for }s=0.
\end{cases}
\end{align*}
Here $[\cdot,\cdot]_{\theta}$ is the complex interpolation method \cite[Example~I.2.4.2]{Ama95}. In particular, by \cite[Corollary~7.2(i)]{Ama13}, $H^{k}_p({\M},V)\doteq W^{k}_p({\M},V)$, for $k\in{\N}_{0}$. Analogously, the Sobolev-Slobodeckii spaces are defined as
\begin{align*}
W^{s}_p({\M},V):=(W^{k}_p({\M},V),V),W^{k}_p({\M},V))_{s-k,p},\hspace{1em} k<s<k+1, 
\end{align*}
where $(\cdot,\cdot)_{\theta,p}$ is the real interpolation method, see \cite[Example~I.2.4.1]{Ama95} and \cite[Section~1.3]{Trie78}.

\begin{prop}
\label{S2: Retraction of SB}
Let ${\F}\in\{W_p,H_p\}$. Then $\Re$ is a retraction from $l_p(\boldsymbol{\F}^{s})$ onto ${\F}^{s}({\M},V)$, and $\Rc$ is a coretraction.
\end{prop}

For $k\in{\N}_{0}$, we define
\begin{align*}
BC^{k}({\M},V):=(\{u\in{C^k({\M},V)}:\|u\|_{k,\infty}^{\M}<\infty\},\|\cdot\|_{k,\infty}^{\M}),
\end{align*}
where $\|u\|_{k,\infty}^{\M}:={\max}_{0\leq{i}\leq{k}}\||\nabla^{i}u|_{g}\|_{\infty}$.
\smallskip\\
We also set
\begin{align*}
BC^{\infty}({\M},V):=\bigcap_{k}BC^{k}({\M},V)
\end{align*}
endowed with the conventional projective topology. Then
\begin{center}
$bc^{k}(\M,V):=$ the closure of $BC^{\infty}$ in $BC^{k}$.
\end{center}
Now the H\"older spaces $BC^s(\M,V)$ of order $s$ for some $k<s<k+1$ with $k\in\N_0$ are defined by
\begin{align*}
BC^s({\M},V):=(bc^{k}({\M},V),bc^{k+1}({\M},V))_{s-k,\infty}.
\end{align*}
We define the little H\"older spaces by 
\begin{center}
$bc^{s}({\M},V):=$ the closure of $BC^{\infty}({\M},V)$ in $BC^{s}({\M},V)$.
\end{center}

We denote by 
\begin{align*}
l_{\infty,\uf}(\boldsymbol{bc}^k)
\end{align*}
the linear subspace of $l_{\infty}(\boldsymbol{BC}^k)$ of all $\boldsymbol{u}=(u_{\kappa})_{\kappa}$ such that $\partial^{\alpha}u_{\kappa}$ is uniformly continuous on $\Xk$ for $|\alpha|\leq k$, uniformly with respect to $\kappa\in\mathfrak{K}$. 
Similarly, for any $k<s<k+1$, we denote by 
\begin{align*}
l_{\infty,\uf}(\boldsymbol{bc}^s)
\end{align*}
the linear subspace of $l_{\infty,\uf}(\boldsymbol{bc}^k)$ of all $\boldsymbol{u}=(u_{\kappa})_{\kappa}$ such that 
\begin{align*}
\lim_{\delta\rightarrow 0}\max_{|\alpha|=k}[\partial^{\alpha}u_{\kappa}]^{\delta}_{s-k,\infty}=0, 
\end{align*}
uniformly with respect to $\kappa\in\mathfrak{K}$. Here the seminorm $[\cdot]^{\delta}_{s,\infty}$ for $0<s<1$ and $0<\delta\leq\infty$ is defined by
\begin{align*}
[u]^{\delta}_{s,\infty}:=\sup_{h\in(0,\delta)^m}\frac{\|u(\cdot+h)-u(\cdot)\|_{\infty}}{|h|^s}, \hspace*{1em}[\cdot]_{s,\infty}:=[\cdot]^{\infty}_{s,\infty}.
\end{align*}
\begin{prop}
\label{S2: retraction of H-LH}
$\Re$ is a retraction from $l_{b}(\boldsymbol{\F}^{s})$ onto $\F^{s}({\M},V)$ and $\Rc$ is a coretraction. Here $b=``\infty"$ for $\F=BC$, or $b=``\infty,\uf"$ for $\F=bc$.
\end{prop}

In the following proposition, $(\cdot,\cdot)^0_{\theta,\infty}$ is the continuous interpolation method, see \cite[Example~I.2.4.4]{Ama95} and \cite[Definition~1.2.2]{Lunar95}.
\begin{prop}
\label{interpolation result}
Suppose that $0<\theta<1$, $0\leq{s}_0<s_1$ and $s=(1-\theta)s_0+\theta{s_1}$. Let $\F\in\{W_p,H_p\}$ and $\mathfrak{B}\in\{bc,BC\}$. Then for $s\notin\N$ except in (b)
\begin{enumerate}
\item[(a)] $({\F}^{s_0}({\M},V),{\F}^{s_1}({\M},V))_{\theta,p}\doteq W^{s}_p({\M},V) $.
\item[(b)] $[{\F}^{s_0}({\M},V),{\F}^{s_1}({\M},V)]_{\theta}\doteq H^{s}_p({\M},V)$, \hspace{.5em}$s_0,s_1\in{\N}_0$ when $\F=W_p$.
\item[(c)] $(\mathfrak{B}^{s_0}({\M},V),\mathfrak{B}^{s_1}({\M},V))_{\theta,\infty}\doteq BC^{s}({\M},V)$, \hspace{.5em} $s_0,s_1\notin\N_0$ when $\mathfrak{B}=BC$, or $s_0,s_1\in\N_0$ when $\mathfrak{B}=bc$.
\item[(d)] $(bc^{s_0}({\M},V),bc^{s_1}({\M},V))_{\theta,\infty}^0\doteq bc^{s}({\M},V)$, \hspace{.5em} $s_0,s_1\notin\N_0$, or $s_0,s_1\in\N_0$.
\end{enumerate}
\end{prop}
\begin{prop}
\label{embedding theory}
The following embedding results hold for function spaces over ${\M}$. For $t>s\geq{0}$
\begin{itemize}
\item[(a)] ${\F}^{t}({\M},V)\overset{d}{\hookrightarrow}{\F}^{s}({\M},V)$, where ${\F}\in \{W_{p},H_{p},bc \}$.
\item[(b)] ${\F}^{t}({\M},V)\overset{d}{\hookrightarrow}{\mathfrak{B}}^{s}({\M},V)$, where $\{\F,\mathfrak{B}\}\in \{(BC,bc),(H_p,W_p),(W_p,H_p)\}$.
\end{itemize}
\end{prop}
Assume that $V_j=V^{\sigma_j}_{\tau_j}:=\{T^{\sigma_j}_{\tau_j}\M,(\cdot|\cdot)_g\}$ with $j=1,2,3$ are $\K$-valued tensor bundles on $\M$. By bundle multiplication from $V_1\times V_2$ into $V_3$, denoted by
\begin{align*}
{\mathsf{m}}: V_1\times V_2\rightarrow V_3,\hspace{1em} (v_1,v_2)\mapsto {\mathsf{m}}(v_1,v_2),
\end{align*}
we mean a smooth bounded section $\mathfrak{m}$ of Hom$(V_1\otimes V_2,V_3)$, i.e.,
\begin{align}
\label{section 2: bundle multiplication}
\mathfrak{m}\in BC^{\infty}(\M, \text{Hom}(V_1\otimes V_2,V_3)), 
\end{align}
such that $\mathsf{m}(v_1,v_2):=\mathfrak{m}(v_1\otimes v_2)$. Its point-wise extension from $\Gamma(\M,V_1\oplus V_2)$ into $\Gamma(\M,V_3)$ is defined by:
\begin{align*}
\mathsf{m}(v_1,v_2)(\p):=\mathsf{m}(\p)(v_1(\p),v_2(\p))
\end{align*}
for $v_i\in \Gamma(\M,V_i)$ and $\p \in\M$. We still denote it by ${\mathsf{m}}$. 
\begin{prop}
\label{pointwise multiplication properties}
Let $V_j=V^{\sigma_j}_{\tau_j}:=\{T^{\sigma_j}_{\tau_j}\M,(\cdot|\cdot)_g\}$ with $j=1,2,3$ be tensor bundles. Suppose that $\mathsf{m}:V_1\times V_2\rightarrow V_3$ is a bundle multiplication. Then $[(v_1,v_2)\mapsto \mathsf{m}(v_1,v_2)]$ is a bilinear and continuous map for the following spaces:
\begin{itemize}
\item[(a)] $BC^{t}({\M},V_1)\times{\F}^{s}({\M},V_2)\rightarrow{\F}^{s}({\M},V_3)$, where $t>s\geq{0}$, ${\F}\in\{W_p,H_p\}$.
\item[(b)] $BC^{k}({\M},V_1)\times W^{k}_p({\M},V_2)\rightarrow W^{k}_p({\M},V_3)$, where $k\in{\N}_0$.
\item[(c)] ${\F}^{s}({\M},V_1)\times{\F}^{s}({\M},V_2)\rightarrow{\F}^{s}({\M},V_3)$, where $s\geq 0$, ${\F}\in\{bc,BC\}$.
\end{itemize}
\end{prop}
\medskip

\section{\bf Parameter-Dependent Diffeomorphism}

This section is designated to develop a family of parameter-dependent diffeomorphisms $\{\ttl:(\lambda,\mu)\in\B\}$ acting on tensor fields and differential operators. This family is induced by a truncated translation, and $\mu$ and $\lambda$ denote translation in time and space, respectively. In \cite{EscPruSim03}, J.~Escher, J.~Pr\"uss and G.~Simonett introduce a parameter-dependent technique to study regularity of solutions to parabolic and elliptic equations in Euclidean spaces. An important observation is that the results therein extend well to $E$-valued function spaces. We will employ this technique to establish the family $\ttl$ on a uniformly regular Riemannian manifold. The applications in later sections will prove it a very beneficial tool in the analysis of regularity of solutions to parabolic differential equations on manifolds, especially in proving analyticity of solutions.

\subsection{\bf Definition and Basic Properties}
Suppose that $(\M,g)$ is a uniformly regular Riemannian manifold equipped with a uniformly regular atlas $\mathfrak{A}$. Given any point $\p \in\mathring{\M}$, there is a local chart $(\Ukp,\varphi_{\kp})\in\mathfrak{A}$ containing $\p$. Let $\x:=\varphi_{\kp}(p)$ and $d:={\rm{dist}}(\x,\partial \mathbb{B}^m_{\kp})$.
Henceforth, $\mathbb{B}(x,r)$ always denotes the ball with radius $r$ centered at $x$ in ${\R}^n$. The dimension $n$ of the ball is not distinguished as long as it is clear from the context. We construct a new local patch $(\Uki,\varphi_{\iota})$ around $\p$ such that $\Uki:=\psi_{\kp}(\mathbb{B}(\x,d))$, and $\varphi_{\iota}(q):=\frac{\varphi_{\kp}(q)-\x}{d}$ for $q\in\Ukp$. Then $\varphi_{\iota}(\p)=0\in\R^m$, $\varphi_{\iota}(\Uki)=\Q$. The transition maps between $(\Uki,\varphi_{\iota}) $ and $(\Ukp,\varphi_{\kp})$ satisfy
\begin{align}
\label{S3: generic transition}
\varphi_{\iota}\circ\psi_{\kp}\in C^{\omega}\cap BC^{\infty}(\mathbb{B}(\x,d),\Q),\hspace{.5em}\varphi_{\kp}\circ\psi_{\iota}\in C^{\omega}\cap BC^{\infty}(\Q,\mathbb{B}(\x,d)).
\end{align}
More precisely, $\varphi_{\iota}\circ\psi_{\kp}(y)=\frac{y-\x}{d}$, $\varphi_{\kp}\circ\psi_{\iota}(x)=xd+\x$ with $x\in\Q, y\in \mathbb{B}(\x,d)$. 

Note that the atlas $\tilde{\mathfrak{A}}:=(\mathsf{O}_{\tilde{\kappa}},\varphi_{\tilde{\kappa}})_{\tilde{\kappa}\in\tilde{\mathfrak{K}}}:=\mathfrak{A}\cup(\Uki,\varphi_{\iota})$  remains uniformly regular. Choose $\varepsilon_0>0$ small such that $5\varepsilon_0<1$ and set
\begin{align*}
B_i:=\mathbb{B}^{m}(0,i\varepsilon_{0}),\hspace*{1em}\text{ for }i=1,2,...,5.
\end{align*}
We may assume that $\zeta|_{B_5}\equiv 1$. Choose two cut-off functions on ${\Qi}$:
\begin{itemize}
\item  $\chi\in{\mathcal{D}(B_{2},[0,1])}$ such that $\chi|_{\bar{B}_1}\equiv 1$. We write $\chi_{\iota}={\kb}\chi$.
\item  ${\Xo}\in\mathcal{D}(B_5,[0,1])$ such that ${\Xo}|_{\bar{B}_4}\equiv{1}$. We write ${\Xt}={\kb}{\Xo}$.
\end{itemize}
Define a rescaled translation on ${\Qi}$ for ${\mu}\in{\B}\subset{\R}^m$ with $r$ sufficiently small:
\begin{align*}
\theta_{\mu}(x):=x+\chi{(x)}\mu,\hspace{1em} x\in{\Qi}.
\end{align*}
Some properties of $\theta_{\mu}$ are listed below without giving proofs. The reader may find more details in \cite[Section~2]{EscPruSim03}. For sufficiently small $r>0$ and any $\mu,\mu_0\in{\B}$,
\begin{itemize}
\item[(T1)] $\theta_{\mu}(B_3)\subset B_3$, and $\theta_{\mu}(\bar{B}_3)\subset \bar{B}_3$.
\item[(T2)] $|\theta_{\mu}(x)-\theta_{\mu_{0}}(y)|\leq{\gamma|x-y|+|\mu-{\mu_{0}}|}$,\hspace{.5em} for some $\gamma>1$, $\forall\hspace{.2em}x,y\in{\Qi}$.
\item[(T3)] $\theta_{\mu}\in{{\Dfi}(U)}$,\hspace{.5em} for any open subset ${U}$ containing $\bar{B}_3$.
\end{itemize}
The truncated shift $\theta_{\mu}$ induces a transformation $\Theta_{\mu}$ on ${\M}$ by:
\begin{align*}
\Theta_{\mu}(q)=
\begin{cases}
\psi_{\iota}(\theta_{\mu}(\varphi_{\iota}(q))) \hspace{1em}&q\in{\Uki},\\
q &q\notin{\Uki}.
\end{cases}
\end{align*}
Based on (T3), it is evident that $\Theta_{\mu}\in{\Dfi}({\M})$ for $\mu\in{\B}$ with sufficiently small $r$. We can find an explicit global expression for ${\ttm}$. 
Given $u\in \Gamma({\M},V)$,
\begin{align*}
{\ttm}u={\kb}{\tm}{\kf}(\Xt u)+(1_{\M}-{\Xt})u.
\end{align*}
Likewise, we can express ${\ttmi}$ by
\begin{align*}
{\ttmi}={\kb}{\theta^{\mu}_{\ast}}{\kf}(\Xt u)+(1_{\M}-{\Xt})u.
\end{align*}
Assuming that ${\F}\in\{bc,BC,W_p,H_p\}$, we define a subspace of ${\F}^{s}(\M,V)$ by
\begin{align*}
{\F}_{\cp}^{s,\M}={\F}_{\cp}^{s}(\M,V):=\{u\in{{\F}^{s}(\M,V)}: \supp(u)\subset{\psi_{\iota}(\bar{B}_4)}\}. 
\end{align*}
In the case ${\F}\in\{W_p,H_p\}$, it is understood that $u$ has a representative satisfying the given condition. Note that ${\F}_{\cp}^{s,\M}$ is closed in ${\F}^{s}(\M,V)$. Hence, ${\F}_{\cp}^{s,\M}$ is a Banach space endowed with the induced norm from ${\F}^{s}(\M,V)$. The Banach spaces 
\begin{align*}
{\F}_{\cp}^{s}={\F}_{\cp}^{s}({\R}^m,E):=\{u\in{{\F}^{s}({\R}^m,E)}: \supp(u)\subset{\bar{B}_4}\} 
\end{align*}
are defined alike. The following lemma enables us to transfer the properties of ${\tm}$ to the transformation ${\ttm}$, and hence plays a key role in the sequel. 
\goodbreak
\begin{lem}
\label{main lemma 3.1}
Let ${\F}\in \{bc,BC,W_p,H_p\}$. Then 
\begin{align*}
{\kb}\in{{\Lis}({\F}_{\cp}^{s},{\F}_{\cp}^{s,\M})},\hspace{.5em}\text{ with }\hspace{.5em}[{\kb}]^{-1}={\kf}.
\end{align*}
\end{lem}
\begin{proof}
Clearly, we have that ${\kb}{\kf}={\id}_{{\F}_{\cp}^{s,\M}}$ and ${\kf}{\kb}={\id}_{{\F}_{\cp}^{s}}$. It follows from the point-wise estimate \cite[Lemma 3.1(iv)]{Ama13} and (R4) that with $r,\sigma,\tau\in\N_0$
\begin{align}
\label{important pt estimate}
&\sum_{i=0}^r{\kft}|\nabla^i a|_g=\sum_{i=0}^r |{\kft}\nabla^i a|_{{\kft}g}\sim \sum_{i=0}^r |{\kft}\nabla^i a|\sim \sum_{|\alpha|\leq r}|\partial^{\alpha}[{\kft}a]|
\end{align}
for any $a\in\mathcal{T}^{\sigma}_{\tau}\M$ and $\tilde{\kappa}\in\tilde{\mathfrak{K}}$, where $|\cdot|$ denotes the Euclidean norm. By \eqref{important pt estimate}, we conclude that for $k\in{\N}_{0}$ 
\begin{align*}
{\kft}\in\mathcal{L}({\F}^{k}(\M,V),{\F}^{k}(\mathbb{B}^m_{\tilde{\kappa}},E)), \hspace{1em} {\F}\in\{bc,BC,W_p\}.
\end{align*}
The case that $\F=bc$ follows from a density argument based on Proposition~\ref{embedding theory}. We can fill in the non-integer $s$ by Proposition~\ref{interpolation result} and interpolation theory, i.e.,
\begin{align}
\label{S3: kfk is bdd}
{\kft}\in\mathcal{L}({\F}^{s}(\M,V),{\F}^{s}(\Qk,E)), \hspace{1em} {\F}\in\{bc,BC,W_p, H_p\}.
\end{align}
See \cite[Theorem~3.1.2 and 4.1.2]{Berg76}. It implies that ${\kf}\in{\mathcal{L}({\F}_{\cp}^{s,\M},{\F}_{\cp}^{s})}$. Now the assertion is an immediate consequence of the open mapping theorem. 
\end{proof}
\begin{remark}
\label{explanation of the main lemma}
Recall that $\F^{s,\M}_{\cp}$ is a closed subspace of $\F^{s}(\M,V)$. We can identify ${\kb}$ as a map into ${\F}^{s}(\M,V)$, that is to say, ${\kb}\in{\mathcal{L}({\F}_{\cp}^{s},{\F}^{s}(\M,V))}$.
\end{remark}
\begin{prop}
\label{ttm is a diffeomorphism}
\begin{itemize}
\item[]{\phantom{ some  text to complete some  } }
\item[(a)] Suppose that ${\F}\in \{bc,BC,W_{p},H_{p}\}$. Then
\begin{align*}
{\ttm}\in{{\Lis}({\F}^{s}(\M,V))},\text{ and } [{\ttm}]^{-1}={\ttmi}.
\end{align*}
Moreover, $\|{\ttm}\|_{\mathcal{L}({\F}^{s}(\M,V))}\leq{M}$ for some $M>0$ and any $\mu\in{\B}$.
\item[(b)] Suppose that $s\geq{0}$ for ${\F}\in \{bc,W_{p},H_{p}\}$, or $s\in{\N}$ for ${\F}=BC$. Then
\begin{align*}
[\mu\mapsto\ttm u]\in C({\B},{\F}^{s}(\M,V)),\hspace{1em} u\in {\F}^{s}(\M,V).
\end{align*}
\end{itemize}
\end{prop}
\begin{proof}
(i) The assertion that ${\ttm}{\ttmi}={\ttmi}{\ttm}={\id}_{\F^{s}(\M,V)}$ is a straightforward consequence of the definitions of ${\ttm}$, ${\ttmi}$ and (T1).
\smallskip\\
(ii) By the open mapping theorem, it is sufficient to show that $\|{\ttm}\|_{\mathcal{L}({\F}^{s}(\M,V))}$ is uniformly bounded with respect to ${\mu\in{\B}}$.
\smallskip\\
Given $f\in BC^{\infty}(\M)$, we define a multiplier operator $\mathsf{m}_{f}: \F^s(\M,V)\rightarrow\F^s(\M,V)$ by
\begin{align*}
\mathsf{m}_{f}: u\mapsto fu.
\end{align*}
By Proposition~\ref{pointwise multiplication properties}, we infer that there exists a constant $M_1$ such that
\begin{align}
\label{S3: multiplication-varsigma 1}
\mathsf{m}_{\varsigma},\mathsf{m}_{1_{\M}-\varsigma}\in\L(\F^s(\M,V)),\hspace*{1em} \|{\varsigma}\|_{\L(\F^s(\M,V))}+\|(1_{\M}-\varsigma)\|_{\L(\F^s(\M,V))}\leq M_1.
\end{align}
Henceforth, we always identify the multiplication operators $\mathsf{m}_{\varsigma}$ and $\mathsf{m}_{1_{\M}-\varsigma}$ with $\varsigma$ and $1_{\M}-\varsigma$, respectively. By Lemma~\ref{main lemma 3.1}, we can find a constant $M_2>0$ such that
\begin{align}
\label{S3: cp-estimate}
\|{\kf}\|_{\L(\F^{s,\M}_{\cp},\F^s_{\cp})}+\|{\kb}\|_{\L(\F^s_{\cp},\F^s(\M,V))}\leq M_2.
\end{align}
Note that $M_1$, $M_2$ depend only on the choice of $\F^s$. On the other hand, by \cite[Proposition~2.4(a)]{EscPruSim03}, there exists a uniform constant $M_3$ with respect to $\mu\in{\B}$ such that
\begin{align*}
\|{\tm}\|_{\mathcal{L}({\F}_{\cp}^{s},{\F}^{s}(\R^m,E))}\leq M_3.
\end{align*}
(T1) implies that ${\tm}({\F}_{\cp}^{s})\subset{{\F}_{\cp}^{s}}$. Since ${\ttm}u=(1_{\M}-\varsigma)u+{\kb}{\tm}{\kf}{\varsigma}u$, there is a uniform constant $M$ such that
\begin{align*}
\|{\ttm}\|_{\L(\F^s(\M,V))}\leq \|1_{\M}-\varsigma\|_{\L(\F^s(\M,V))}+\|{\kb}\circ {\tm}\circ {\kf}\circ {\varsigma}\|_{\L(\F^s(\M,V))}\leq M
\end{align*}
for all $\mu\in\B$. The case $\F=bc$ follows by a density argument.
\smallskip\\
(iii) Pick $\mu,\mu_{0}\in{\B}$. Then
\begin{align*}
\notag\|{\ttm}u-{\Theta}^{\ast}_{\mu_{0}}u\|_{{\F}^{s}}^{\M}&=\|{\kb}({\tm}-\theta^{\ast}_{\mu_{0}}){\kf}{\varsigma}u\|_{{\F}^{s}}^{\M} \leq{M_{2}\|({\tm}-\theta^{\ast}_{\mu_{0}}){\kf}{\varsigma}u\|_{{\F}_{\cp}^{s}}}.
\end{align*}
By Lemma~\ref{main lemma 3.1} and \eqref{S3: multiplication-varsigma 1}, ${\kf}{\varsigma}u\in{{\F}_{\cp}^{s}}$. \cite[Proposition~2.4(b)]{EscPruSim03} implies that ${\ttm}u\rightarrow {\Theta}^{\ast}_{\mu_{0}}u$ in $\F^s(\M,V)$ as $\mu\rightarrow{\mu_{0}}$. This completes the proof.
\end{proof}

\subsection{\bf Higher Regularity}
In this subsection, we will show that regularity of the map $[{\mu}\mapsto {\ttm}u]$ can be inherited from the local smoothness of $u$ near $\p \in\mathring{\M}$. 

Given any open subset ${\U}\subset{\M}$ containing $\p$, by choosing $\varepsilon_0$ small enough, we can always achieve $\psi_{\iota}(B_{3})\subset  \U$. Without loss of generality, we may assume that $\U \subset \psi_{\iota}({B_{4}})$, since $u\in C^k(\U,V)$ implies $u\in C^k(\U^{\prime},V)$ for any $\U^{\prime}\subset \U$. Denote by $V_{\U}$ the restriction of $V$ on $\U$. We say that a function $u\in L_{1,loc}(\U,V_{\U})$ belongs to $\F^{s}(\U,V)$ if ${\kf}u\in \F^s(\varphi_{\iota}(\U),E)$, or equivalently ${\kfp}u\in \F^s(\varphi_{\kp}(\U),E)$ by \eqref{S3: generic transition}.
\smallskip\\
Suppose that $\M$ is a $C^{\omega}$-uniformly regular Riemannian manifold. We say a tensor field $u\in C^{\omega}({\U},V)$ if ${\kf}u\in{C^{\omega}({\varphi_{\iota}({\U})},E)}$, or equivalently ${\kfp}u\in C^{\omega}(\varphi_{\kp}(\U),E)$. Hereafter, it is understood that in the case $k=\omega$, the manifold $\M$ is always assumed to be $C^{\omega}$-uniformly regular.
\begin{theorem}
\label{differentiability of ttm on all function spaces}
Suppose that $u\in{C^{l+k}({\U},V)\cap{{\F}^{s}(\M,V)}}$, where $s\in{[0,l]}$ for ${\F}\in \{bc,W_{p},H_{p} \}$, or $s=l$ for $\F=BC$ with $l\in{\N}_{0}$, $k\in\N\cup\{\infty,\omega\}$. Then
\begin{align*}
[\mu\mapsto{\ttm}u]\in{C^{k}({\B},{\F}^{s}(\M,V))}.
\end{align*}
Moreover, $\partial^{\alpha}_{\mu}[{\ttm}u]={\kb}{\chi^{|\alpha|}}{\tm}\partial^{\alpha}{\kf}{\varsigma}u$.
\end{theorem}
\begin{proof}
As in Proposition~\ref{ttm is a diffeomorphism}, one checks that ${\kf}{\varsigma}u\in C^{l+k}(\varphi_{\iota}({\U}),E)\cap \F^s_{\cp}$ for every $u\in C^{l+k}({\U},V)\cap\F^s(\M,V)$.
We conclude from \cite[Theorem~3.3]{EscPruSim03} that
\begin{align*}
[\mu\mapsto {\tm}{\kf}{\varsigma}u]\in C^k(\B, \F^s_{\cp}).
\end{align*} 
By Remark~\ref{explanation of the main lemma}, ${\kb}\in \L(\F^s_{\cp},\F^s(\M,V))$ and thus is real analytic. This proves the first part of the assertion. 
\smallskip\\
Pick $\mu\in{\B}$ and choose $\varepsilon>0$ so small that $\mu+he_{j}\in{\B}$ for all $h\in{(-\varepsilon,\varepsilon)}$ and $j\in\{1,2,...,m\}$.
\begin{align}
\label{S3: derivative of ttm}
\notag\lim\limits_{\varepsilon\rightarrow 0}\frac{1}{h}({\Theta}^{\ast}_{\mu+he_{j}}u-{\ttm}u)&={\kb}\lim\limits_{\varepsilon\rightarrow 0}\frac{1}{h}({\theta}^{\ast}_{\mu+he_j}{\kf}{\varsigma}u-{\tm}{\kf}{\varsigma}u)\\
&={\kb}{\chi}{\tm}\partial_j{\kf}{\varsigma}u,
\end{align}
which converges in $\F^s(\M,V)$. The first equality follows from the boundedness of ${\kb}$. Meanwhile, it follows from \cite[Proposition~3.2]{EscPruSim03} and the fact ${\kf}{\varsigma}u\in C^{l+k}(\varphi_{\iota}(\U),E)\cap BU\!C^{l}(B_3,E)$ that 
\begin{align*}
\lim\limits_{\varepsilon\rightarrow 0}\frac{1}{h}({\theta}^{\ast}_{\mu+he_j}{\kf}{\varsigma}u-{\tm}{\kf}{\varsigma}u)
\end{align*}
converges in $BC^{l}_{\cp}$. Here $BU\!C^l(\Q,E)$ is the closed linear subspace of $BC^l(\Q,E)$ consisting of $u\in BC^{l}(\Q,E)$ such that $\partial^{\alpha}u$ is uniformly continuous for all $|\alpha|\leq l$. Owing to the embedding $BC^{l}_{\cp}\hookrightarrow \F^s_{\cp}$ and Lemma~\ref{main lemma 3.1}, the convergence of \eqref{S3: derivative of ttm} in $\F^s(\M,V)$ now is straightforward. The rest of the assertion follows by induction.
\end{proof}

The following inverse of Theorem~\ref{differentiability of ttm on all function spaces} is of indispensable character in analyzing regularity of solutions to differential equations.
\begin{theorem}
\label{main theorem-iff}
Let $k\in\N\cup\{\infty,\omega\}$. Suppose that $u\in{BC(\M,V)}$. Then $u\in{C^{k}({\M},V)}$ iff for any $\p \in\mathring{\M}$, there exists $r=r(\p)>0$ and a corresponding family of parameter-dependent diffeomorphisms $\{{\ttm}:r\in\B\}$ such that
\begin{align*}
[\mu\mapsto{\ttm}u]\in{C^{k}({\B},BC(\M,V))}.
\end{align*}
\end{theorem}
\begin{proof}
The ``only if" part has been established in Theorem~\ref{differentiability of ttm on all function spaces}. So we only deal with the ``if" part. For every $\p \in\mathring{\M}$, consider the evaluation map 
\begin{align*}
\tilde{\gamma}_{\p}: BC(\M,V)\rightarrow{E},\hspace{1em} u\mapsto{\kf}u(\varphi_{\iota}(\p))={\kf}u(0).
\end{align*}
Then $\tilde{\gamma}_{\p}\in{\mathcal{L}(BC(\M,V),E)}$, and thus is real analytic. Moreover, for small enough $r$ it holds that
\begin{align*}
\tilde{\gamma}_{\p}({\ttm}u)=({\kf}{\kb}{\tm}{\kf}{\varsigma}u)(0)+({\kf}(1_{\M}-\varsigma)u)(0)={\kf}u(\mu).
\end{align*}
This implies that $[\mu\mapsto {\kf}u(\mu)]\in{C^{k}({\B},E)}$. By \eqref{S3: generic transition}, it yields
\begin{align*}
[y\mapsto {\kfp}u(y)]\in{C^{k}(\mathbb{B}(\x,rd),E)}. 
\end{align*}
Hence for any $\p \in\mathring{\M}$, we can find a local chart $(\psi_{\kp}(\mathbb{B}(\x,rd)),\varphi_{\kp})$ with $\kp\in\mathfrak{K}$ such that $\varphi_{\kp}(\p)=\x$ and ${\kfp}u\in{C^{k}(\mathbb{B}(\x,rd),E)}$. Therefore, $u\in{C^{k}({\M},V)}$ by definition.
\end{proof}

\subsection{\bf Differential Operators}
Let $l\in{\N}_0$. A linear operator $\mathcal{A}:C^{\infty}({\M},V)\rightarrow \Gamma({\M},V)$ is called a linear differential operator of order $l$ on ${\M}$ if we can find $\boldsymbol{\mathfrak{a}}=(a^r)_r\in \prod_{r=0}^l \Gamma(\M, V^{\sigma+\tau+r}_{\tau+\sigma})$ such that
\begin{align}
\label{section 2: globally-defined diff-op}
\mathcal{A}=\mathcal{A}(\boldsymbol{\mathfrak{a}}):=\sum\limits_{r=0}^l \ev(a^r,\nabla^r \cdot).
\end{align}
Here $\ev:\Gamma(\M, V^{\sigma+\tau+r}_{\tau+\sigma}\times V^{\sigma}_{\tau+r})\rightarrow \Gamma(\M, V^\sigma_\tau)$ is the complete contraction, see \cite[Section~2.3]{ShaoSim13} for details. In every local chart $({\Uk},\varphi_{\kappa})$, there exists some linear differential operator defined on $\Qk$
\begin{align*}
\mathcal{A}_{\kappa}(x,\partial):=\sum_{|\alpha|\leq{l}}a^{\kappa}_{\alpha}(x)\partial^{\alpha},\hspace*{.5em}\text{ with }a^{\kappa}_{\alpha}\in \L(E)^{\Qk},
\end{align*}
called the local representation of $\mathcal{A}$ in $(\Uk,\varphi_{\kappa})$, such that for any $u\in C^{\infty}({\M},V)$
\begin{align}
\label{S3:local exp}
{\kfk}(\mathcal{A}u)=\mathcal{A}_{\kappa}({\kfk}u).
\end{align}
Proposition~\ref{pointwise multiplication properties} empowers us to extend \cite[Corollary~2.10]{ShaoSim13} to: 
\begin{prop}
\label{differential operator properties}
Suppose that $\mathcal{A}:=\mathcal{A}(\boldsymbol{\mathfrak{a}})$ is a linear differential operator of order $l$ on ${\M}$ such that $\boldsymbol{\mathfrak{a}}=(a^r)_r\in \prod_{r=0}^l BC^{t}(\M,V^{\sigma+\tau+r}_{\tau+\sigma})$, or equivalently $(a_{\alpha}^{\kappa})_{\kappa}\in l_{\infty}(\boldsymbol{BC}^{t}(\Qk,\L(E)))$ for all $|\alpha|\leq l$ and $\kappa\in\mathfrak{K}$. Then 
\begin{align*}
\mathcal{A}\in\mathcal{L}({\F}^{s+l}(\M,V),{\F}^{s}(\M,V)), 
\end{align*}
where $t>s$ for ${\F}\in \{bc,W_{p},H_{p} \}$, or $t\geq s$ for ${\F}=BC$, or $t=s\in\N_0$ for $\F=W_p$.
\end{prop}
The parameter-dependent family of diffeomorphisms, i.e., $\{{\ttm}:\mu\in{\B}\}$, generates a parameter-dependent family of differential operators, $\{{\Al}:\mu\in{\B}\}$, on ${\M}$, given by
\begin{align*}
{\Al}:={\ttm}\mathcal{A}{\ttmi}.
\end{align*}
We shall show that regularity of the coefficients $a_{\alpha}^{\kappa}$ translates into the smoothness of the map $[\mu\mapsto{\Al}]$.
\begin{prop}
\label{differentiability of differential operators-variable}
Let $n\in{\N}_{0}$, $k\in{\N}_{0}\cup\{\infty,\omega\}$, and $\p\in\mathring{\M}$.
\begin{itemize}
\item[(a)] Let $s\in[0,n]$ if ${\F}\in\{BC,W_{p},H_{p}\}$, or $s\in[0,n)$ if $\F=bc$. Suppose that $\mathcal{A}:=\mathcal{A}(\boldsymbol{\mathfrak{a}})$ is a linear differential operator on $\M$ of order $l$ satisfying
\begin{align*}
\boldsymbol{\mathfrak{a}}=(a^r)_r\in \prod_{r=0}^l BC^{n}(\M,V^{\sigma+\tau+r}_{\tau+\sigma})\cap C^{n+k}(\U,V^{\sigma+\tau+r}_{\tau+\sigma}),
\end{align*}
or equivalently for all $|\alpha|\leq l$ and $\kappa\in\mathfrak{K}$, $(a_{\alpha}^{\kappa})_{\kappa}\in l_{\infty}(\boldsymbol{BC}^{n}(\Qk,\L(E)))$ and $a_{\alpha}^{\kp}\in C^{n+k}(\varphi_{\kp}({\U}),\L(E))$ with $\U$ defined as in Section~3.2. Then
\begin{align*}
[\mu\mapsto{\Al}]\in{C^{k}({\B},\mathcal{L}({\F}^{s+l}(\M,V),{\F}^{s}(\M,V)))}.
\end{align*}
\item[(b)] Let $s\in[0,n]$. Suppose that $\mathcal{A}$ satisfies
\begin{align*}
\boldsymbol{\mathfrak{a}} \in \prod_{r=0}^l bc^{s}(\M,V^{\sigma+\tau+r}_{\tau+\sigma})\cap C^{n+k}(\U,V^{\sigma+\tau+r}_{\tau+\sigma}),
\end{align*}
or equivalently for all $|\alpha|\leq l$ and $\kappa\in\mathfrak{K}$, $(a_{\alpha}^{\kappa})_{\kappa}\in l_{\infty,\uf}(\boldsymbol{bc}^{s}(\Qk,\L(E)))$ and $a_{\alpha}^{\kp}\in C^{n+k}(\varphi_{\kp}({\U}),\L(E))$ with $\U$ defined as in Section~3.2. Then
\begin{align*}
[\mu\mapsto{\Al}]\in{C^{k}({\B},\mathcal{L}({bc}^{s+l}(\M,V),{bc}^{s}(\M,V)))}.
\end{align*}
\end{itemize}
\end{prop}
\begin{proof}
Note that $\mathcal{A}_{\iota}={\kf}{\kbp}\mathcal{A}_{\kp}{\kfp}{\kb}$.
The conditions of (a) and \eqref{S3: generic transition} imply that $a_{\alpha}^{\iota}\in C^{n+k}(\varphi_{\iota}({\U}),\L(E))\cap BC^{n}(\Q,\L(E))$.
For any $u\in{{\F}^{s+l}(\M,V)}$, 
\begin{align*}
{\Am}u&={\ttm}\mathcal{A}{\kb}\theta^{\mu}_{\ast}{\kf}{\varsigma}u+{\ttm}\mathcal{A}(1_{\M}-\varsigma)u\\
&=\underbrace{{\kb}{\tm}{\kf}{\varsigma}\mathcal{A}{\kb}\theta^{\mu}_{\ast}{\kf}{\varsigma}u+(1_{\M}-\varsigma)\mathcal{A}{\kb}\theta^{\mu}_{\ast}{\kf}{\varsigma}u}_{{\Af}u}\\
&\quad+\underbrace{{\kb}{\tm}{\kf}{\varsigma}\mathcal{A}(1_{\M}-\varsigma)u+(1_{\M}-\varsigma)\mathcal{A}(1_{\M}-\varsigma)u}_{{\As}u}.
\end{align*}
Now we compute ${\Af}u$ and ${\As}u$ separately. Observe that ${\kb}\theta^{\mu}_{\ast}{\kf}(1_{\M}-\varsigma)v=(1_{\M}-\varsigma)v$ for $v$ compactly supported in $\Uki$. We hence infer that
\begin{align}
\label{S3: eq-Af}
{\Af}u &={\kb}{\tm}{\kf}({\varsigma}+(1_{\M}-\varsigma))\mathcal{A}{\kb}{\theta^{\mu}_{\ast}}{\kf}{\varsigma}u\\
\notag&={\kb}{\tm}\mathcal{A}_{\iota}{\kf}{\kb}{\theta^{\mu}_{\ast}}{\kf}{\varsigma}u={\kb}{\tm}\mathcal{A}_{\iota}{\theta^{\mu}_{\ast}}{\kf}{\varsigma}u.
\end{align}
The second equality follows from \eqref{S3:local exp}. A similar argument to \eqref{S3: eq-Af} implies
\begin{align*}
{\As}u&={\varsigma}\mathcal{A}(1_{\M}-\varsigma)u+(1_{\M}-\varsigma)\mathcal{A}(1_{\M}-\varsigma)u
=\mathcal{A}(1_{\M}-\varsigma)u.
\end{align*}
Proposition~\ref{differential operator properties} and \eqref{S3: multiplication-varsigma 1} yield ${\As}\in\mathcal{L}({\F}^{s+l}(\M,V),{\F}^{s}(\M,V))$, thus
\begin{align*}
[\mu\mapsto{\As}]\in{C^{\omega}({\B},\mathcal{L}({\F}^{s+l}(\M,V),{\F}^{s}(\M,V)))}.
\end{align*}
Put $N=m^{\sigma}\times m^{\tau}$. Since the space ${\F}^{s+l}(\Q,E)$ can be identified as $\F^{s+l}(\Q)^N$, we can rewrite $\mathcal{A}_{\iota}$ in matrix form as $\mathcal{A}_{\iota}=(\mathcal{A}_{\iota,ij})_{1\leq i,j \leq N}$. Each $\mathcal{A}_{\iota,ij}$ is a linear differential operator of order at most $l$ with $C^{n+k}(\varphi_{\iota}({\U}))\cap {BC}^{n}(\Qi)$-coefficients acting on $\F^{s+l}(\Qi)$. 
By \cite[Theorem~4.2]{EscPruSim03}, 
\begin{align}
\label{section 3 eq-const coef-scalar}
[\mu\mapsto{\tm}\mathcal{A}_{\iota,ij}{\theta^{\mu}_{\ast}}]\in{C^{k}({\B},\mathcal{L}({\F}_{\cp}^{s+l}(\R^m),{\F}_{\cp}^{s}(\R^m)))}. 
\end{align}
\eqref{section 3 eq-const coef-scalar} implies that
\begin{align*}
[\mu\mapsto{\tm}\mathcal{A}_{\iota}{\theta^{\mu}_{\ast}}]\in{C^{k}({\B},\mathcal{L}({\F}_{\cp}^{s+l},{\F}_{\cp}^{s}))}.
\end{align*}
In virtue of Lemma~\ref{main lemma 3.1}, Remark~\ref{explanation of the main lemma}, formula~\eqref{S3: multiplication-varsigma 1} and the independence of the map $[u\mapsto{\kf}{\varsigma}u]$ and ${\kb}$ on $\mu$, we immediately conclude that
\begin{align*}
[\mu\mapsto \Af]\in C^{k}(\B,\L({\F}^{s+l}(\M,V),{\F}^{s}(\M,V))).
\end{align*}
This establishes the assertion (a). Statement (b) follows from an analogous argument in light of \cite[Proposition~2.9]{ShaoSim13}.
\end{proof}

\subsection{\bf Time Dependence}
Hereafter, we involve the time variable in our study. Let $I=[0,T]$, $T>0$. Assume that $J\subset\mathring{I}$ is an open interval and $t_{0}\in{J}$ is a fixed point. Choose $\varepsilon_{0}$ so small that $\mathbb{B}(t_{0},3\varepsilon_{0})\subset{J}$. Pick an auxiliary function 
\begin{align*}
\xi\in\mathcal{D}(\mathbb{B}(t_{0},2\varepsilon_{0}),[0,1])\hspace{1em}\text{ with }\hspace{.5em} \xi|_{\mathbb{B}(t_{0},\varepsilon_{0})}\equiv{1}. 
\end{align*}
A straightforward modification of the construction in Section 3.1 now engenders a parameter-dependent transformation in terms of the time variable:
\begin{align*}
\varrho_{\lambda}(t):=t+\xi(t)\lambda,\hspace{.5em}\text{ for any }t\in{I} \text{ and } \lambda\in{\R}.
\end{align*}
It is not hard to deduce that for sufficiently small $r>0$ we have
\begin{align*}
\varrho_{\lambda}\in{\Dfi}(J)\cap{\Dfi}(I),\hspace{.5em} \text{ for any } \lambda\in\mathbb{B}(0,r).
\end{align*}
Now we define a parameter-dependent transformation involving both time and space variables, given by
\begin{align*}
{\theta_{\lambda,\mu}}(t,x):=(t+\xi(t)\lambda,x+\xi(t)\chi(x)\mu), \text{ for } (t,x)\in{J\times{U}} \text{ and }(\lambda,\mu)\in{\R}^{m+1},
\end{align*}
where $\chi$ is defined in Section 3.1 and $U$ is a given open subset in ${\R}^m$ containing $\mathbb{B}(0,3\varepsilon)$. It is a simple matter to show that ${\theta_{\lambda,\mu}}\in{\Dfi}(J\times{U})$ for any $(\lambda,\mu)\in\mathbb{B}^{m+1}(0,r)$ for sufficiently small $r$. Here and in the following, I will not distinguish between ${\B}$, $\mathbb{B}^{m}(0,r)$ and $\mathbb{B}^{m+1}(0,r)$. As long as the dimension of the ball is clear from the context, we always simply write them ${\B}$.

For $v:I\times U\rightarrow{E}$, the parameter-dependent diffeomorphism can be expressed as
\begin{center}
$\tilde{v}_{\lambda,\mu}(t,\cdot):={\tl}v(t,\cdot)=\tilde{T}_{\mu}(t){\rh}v(t,\cdot)$,\hspace{1em} where $\tilde{T}_{\mu}(t):=\theta^{\ast}_{\xi(t)\mu}$, for $t\in{I}$.
\end{center}
As before, we define the induced parameter-dependent transformation on $I\times{\M}$ as follows. Given a function $u:I\times{\M}\rightarrow{V}$, we set
\begin{align*}
{\ut}(t,\cdot):={\ttl}u(t,\cdot):={\tu}(t){\rh}u(t,\cdot),
\end{align*}
where ${\tu}(t):={\Theta}^{\ast}_{\xi(t)\mu}$ and $(\lambda,\mu)\in{\B}$. It is important to note that ${\ut}(0,\cdot)=u(0,\cdot)$ for any function $u$ and any $(\lambda,\mu)$.

In order to show smoothness of the family of parameter-dependent transformations $\{{\ttl}:(\lambda,\mu)\in\B\}$, I will first quote:
\begin{lem}{\cite[Lemma~5.1]{EscPruSim03}}
\label{quoting lemma}
Let $X$ be a Banach space. Suppose that
\begin{center}
$[\mu\mapsto{f}(\mu)]\in{C^{k}({\B},X)}$,\hspace{.5em} for\hspace{.2em} $k\in{\N}\cup\{\infty,\omega\}$.
\end{center}
Let $F(\mu)(t):=f(\xi(t)\mu)$\hspace{.2em} for \hspace{.2em}$\mu\in{\B}$ and $t\in{I}$. Then we have
\begin{center}
$[\mu\mapsto{F(\mu)}]\in{C^{k}({\B},C(I,X))}$.
\end{center}
\end{lem}
\goodbreak
\begin{prop}
\label{differentiability involving time}
Let $l\in{\N}_{0}$, $k\in{\N}_{0}\cup\{\infty,\omega\}$, and $\p\in\mathring{\M}$. 
\begin{itemize}
\item[(a)] Suppose that $u\in{C^{l+k}({\U},V)\cap{\F}^{s}(\M,V)}$, where either $s\in[0,l]$ if ${\F}\in\{bc,W_{p},H_{p}\}$, or $s=l$ if $\F=BC$. Then we have
\begin{align*}
[\mu\mapsto{\tu}u]\in{C^{k}({\B},C(I,{\F}^{s}(\M,V)))}.
\end{align*}
Moreover, $\partial^{\alpha}_{\mu}[{\tu}u]={\kb}(\xi\chi)^{|\alpha|}\tilde{T}_{\mu}\partial^{\alpha}{\kf}{\varsigma}u$\hspace{.2em} for \hspace{.2em}$|\alpha|\leq{k}$.
\item[(b)] Let $n\in{\N}_{0}$. Suppose that $\mathcal{A}:=\mathcal{A}(\boldsymbol{\mathfrak{a}})$ is a linear differential operator on $\M$ of order $l$ satisfying 
\begin{align*}
\boldsymbol{\mathfrak{a}}=(a^r)_r\in \prod_{r=0}^l BC^{n}(\M,V^{\sigma+\tau+r}_{\tau+\sigma})\cap C^{n+k}(\U,V^{\sigma+\tau+r}_{\tau+\sigma}),
\end{align*}
or equivalently for all $|\alpha|\leq l$ and $\kappa\in\mathfrak{K}$, $(a_{\alpha}^{\kappa})_{\kappa}\in l_{\infty}(\boldsymbol{BC}^{n}(\Qk,\L(E)))$ and $a_{\alpha}^{\kp}\in C^{n+k}(\varphi_{\kp}({\U}),\L(E))$ with  $\U$ defined as in Section~3.2. Then for $s\in[0,n]$ if ${\F}\in\{BC,W_{p},H_{p}\}$, or $s\in[0,n)$ if $\F=bc$ 
\begin{align*}
[\mu\mapsto{\tu}\mathcal{A}{\tui}]\in{C^{k}({\B},C(I,\mathcal{L}({\F}^{s+l}(\M,V),{\F}^{s}(\M,V))))}.
\end{align*}
\end{itemize}
\end{prop}
\begin{proof}
(a) Set $X={\F}^{s}(\M,V)$, $f(\mu)={\ttm}u$. Now (a) is a direct consequence of Proposition~\ref{ttm is a diffeomorphism}, Theorem \ref{differentiability of ttm on all function spaces} and Lemma \ref{quoting lemma}.

(b) Set $X=\mathcal{L}({\F}^{s+l}(\M,V),{\F}^{s}(\M,V))$ and $f(\mu)={\Al}$. The assertion follows by Proposition \ref{differentiability of differential operators-variable} and Lemma \ref{quoting lemma}.
\end{proof}
For ${\F}\in\{bc,W_{p},H_{p}\}$ and $l\in{\N}$, we put
\begin{align*}
{\ef}(I):=C^{1}(I,{\F}^{s}(\M,V))\cap{C(I,{\F}^{s+l}(\M,V))},
\end{align*}
or 
\begin{align*}
{\ef}(I):=W^{1}_{p}(I,{\F}^{s}(\M,V))\cap{L_{p}(I,{\F}^{s+l}(\M,V))}.
\end{align*}
\begin{prop}
\label{transformed functions}
Let $l\in\N$. Suppose that $u\in {\ef}(I)$. Then ${\ut}\in {\ef}(I)$. Moreover, there exists some ${B}_{\lambda,\mu}$ satisfying
\begin{align}
\label{S3:blm-reg}
[(\lambda,\mu)\mapsto {B}_{\lambda,\mu}]\in{C^{\omega}({\B},C(I,\mathcal{L}({\F}^{s+l}(\M,V),{\F}^{s}(\M,V))))}
\end{align}
such that
\begin{align*}
\partial_{t}[{\ut}]=(1+\xi^{\prime}\lambda){\ttl}u_{t}+{B}_{\lambda,\mu}({\ut}).
\end{align*}
In particular, $B_{\lambda,0}=0$.
\end{prop}
\begin{proof}
(i) Since $I$ is a uniformly regular Riemannian manifold, for any Banach space $X$, an analogous result to Proposition \ref{ttm is a diffeomorphism} holds for $BC^{l}(I,X)$ and $W^{l}_{p}(I,X)$, that is to say, ${\rh}\in{\Lis}({\F}(I,X))$ and $[{\rh}]^{-1}=\varrho^{\lambda}_{\ast}$ for $\lambda\in{\B}$ and ${\F}\in\{BC^{l},W^{l}_{p}\}$. Moreover, there exists $M>0$ such that $\|{\rh}\|_{\mathcal{L}({\F}(I,X))}\leq{M}$ for any $\lambda\in{\B}$. See \cite[Section~2]{EscPruSim03} for more details of the proof. We conclude that
\begin{center}
${\rh}u\in{\ef}(I)$\hspace{.5em} and\hspace{.5em} $\partial_{t}{\rh}u=(1+\xi^{\prime}\lambda){\rh}\partial_{t}u$,\hspace{.5em} for\hspace{.2em} $u\in{\ef}(I)$,
\end{center}

(ii) By definition, $\ut={\kb}\tilde{T}_{\mu}{\kf}{\varsigma}{\rh}u+(1_{\M}-\varsigma){\rh}u$. \eqref{S3: multiplication-varsigma 1} implies that $(1_{\M}-\varsigma){\rh}u\in{\ef}(I)$. For the same reason, one checks that 
\begin{align*}
{\varsigma}{\rh}u\in{\mathbb{E}}_{1}^{\cp}(I)\in\{C^{1}(I,{\F}^{s,\M}_{\cp})\cap{C(I,{\F}^{s+l,\M}_{\cp})},W^{1}_{p}(I,{\F}^{s,\M}_{\cp})\cap{L_{p}(I,{\F}^{s+l,\M}_{\cp})}\}.
\end{align*}
By \cite[Proposition~5.3]{EscPruSim03} and Lemma \ref{main lemma 3.1}, we immediately have that
\begin{align*}
{\kb}\tilde{T}_{\mu}{\kf}{\varsigma}{\rh}u \in\mathbb{E}^{\cp}_{1}(I)\subset{{\ef}(I)},
\end{align*}
which yields ${\ttl}({\ef}(I))\subset{{\ef}(I)}$. 

(iii) With either choice of $\ef(I)$, the time derivative of ${\ut}$ can be computed as
\begin{align*}
&\quad \partial_{t}[{\ut}]={\kb}(\partial_{t}\tilde{T}_{\mu}{\kf}{\varsigma}{\rh}u)+(1_{\M}-\varsigma)\partial_{t}({\rh}u)\\
&={\kb}\tilde{T}_{\mu}{\kf}[{\varsigma}\partial_{t}({\rh}u)]+\sum_{j}{\kb}[\xi^{\prime}\chi\mu_{j}\tilde{T}_{\mu}\partial_{j}({\kf}{\varsigma}{\rh}u)]+(1_{\M}-\varsigma)\partial_{t}({\rh}u)\\
&={\kb}\tilde{T}_{\mu}{\kf}[{\varsigma}(1+\xi^{\prime}\lambda){\rh}u_{t}]+\sum_{j}{\kb}[\xi^{\prime}\chi\mu_{j}\tilde{T}_{\mu}\partial_{j}({\kf}{\varsigma}{\tui}{\ut})]\\
&\hspace*{1em}+(1_{\M}-\varsigma)(1+\xi^{\prime}\lambda){\rh}u_{t}\\
&=(1+\xi^{\prime}\lambda){\ttl}u_{t}+{B}_{\lambda,\mu}({\ut}){.}
\end{align*}
In the last step, by definition of $\tu$, we get 
\begin{align*}
{\kb}\tilde{T}_{\mu}{\kf}[{\varsigma}(1+\xi^{\prime}\lambda){\rh}u_{t}]+(1_{\M}-\varsigma)(1+\xi^{\prime}\lambda){\rh}u_{t}=(1+\xi^{\prime}\lambda){\ttl}u_{t}.
\end{align*}
We can write ${B}_{\lambda,\mu}$ in an explicit way as
\begin{align*}
{B}_{\lambda,\mu}(\cdot)=\sum_{j}{\kb}[\xi^{\prime}\chi\mu_{j}\tilde{T}_{\mu}\partial_{j}({\kf}{\varsigma}{\tui}\cdot)]=\sum_{j}\mu_{j}\xi^{\prime}{\chi_{\iota}}({\tu}\mathcal{A}^{j}{\tui})(\cdot),
\end{align*}
where $\mathcal{A}^{j}$ is a first order linear differential operator compactly supported in $\Uki$ such that $\mathcal{A}_{\iota}^{j}=\zeta\partial_{j}$ and $a_{\alpha}^{\kappa}\in l_{\infty}({\boldsymbol{BC}^{[s]+1}(\mathbb{B},\L(E)))}$. Such $\mathcal{A}^{j}$ can be obtained by means of the recursive construction in \cite[Section~2.3]{ShaoSim13}. So by Proposition \ref{differentiability involving time}, 
\begin{align*}
[\mu\mapsto {\tu}\mathcal{A}^{j}{\tui}]\in{C^{\omega}({\B},C(I,\mathcal{L}({\F}^{s+l}(\M,V),{\F}^{s}(\M,V))))}. 
\end{align*}
Now \eqref{S3:blm-reg} follows from Proposition~\ref{pointwise multiplication properties}.
\end{proof}
\begin{remark}
Any theorem in this section can be formulated with $\M$ being a $C^k$-uniformly regular Riemannian manifold as long as $k$ is no smaller than the highest order of function spaces appearing in that theorem.
\end{remark}

Before closing this Section, we give a proof for the main theorem:
\begin{proof}
(of Theorem~\ref{S1: Theorem}) We can define a new manifold $\mathcal{M}:={I}\times{\M}$. Whenever $\M$ is a $C^k$-uniformly regular Riemannian manifold, $\mathcal{M}$ is also one.  On $\mathcal{M}$, the family $\{\ttl:(\lambda,\mu)\in\B \}$ of diffeomorphisms plays a similar role to $\{\ttm:\mu\in\B \}$ on $\M$. Then the proof of Theorem~\ref{main theorem-iff} can be carried out in the new settings by a minute modification with $\tilde{\gamma}_{\p}$ replaced by $\tilde{\gamma}_{(t_0,\p)}$.
\end{proof}
\begin{rmk}
(a) The above proof implies that  the family of diffeomorphisms $\{\ttl:(\lambda,\mu)\in\B \}$ possesses all the properties established in Section 3.1-3.3.
\smallskip\\
(b) The proof of Lemma~\ref{main lemma 3.1} does not rely on the structure of a uniform regular atlas. This fact shows that the family $\{\ttm:\mu\in\B\}$ can be defined by using any $C^{\infty}$-compatible atlas with only slight modifications so that the results in this section remain valid.
\smallskip\\
(c) Even with the presence of real analyticity, namely, $k=\omega$, in order to establish the theorems in this section, it is sufficient to assume that $\M$ is a uniformly regular Riemannian manifold.
\smallskip\\
In fact, by a well-known theorem of H. Whitney, every $C^{\infty}$-Riemannian manifold admits a compatible real analytic atlas $\hat{\mathfrak{A}}=(\mathsf{O}_{\hat{\kappa}},\varphi_{\hat{\kappa}})_{\hat{\kappa}\in\hat{\mathfrak{K}}}$. We use the convention that $u\in C^{\omega}(\M,V)$ if it is real analytic in terms of $\hat{\mathfrak{A}}$. Around any given $\p \in\mathring{\M}$, we can pick a subset $S\subset\hat{\mathfrak{K}}$ so that $\p\in \mathsf{O}_{\hat{\eta}}$ for some $\hat{\eta}\in S$ and 
\begin{align*}
\{(\Uk,\varphi_{\kappa}):\kappa\in(\mathfrak{K}\setminus\mathfrak{N}(\kp))\}\cup  \{(\mathsf{O}_{\hat{\kappa}},\varphi_{\hat{\kappa}}): \hat{\kappa}\in S\}
\end{align*}
is still a uniformly regular atlas for $\M$, after possibly normalizing the local patches $(\mathsf{O}_{\hat{\kappa}},\varphi_{\hat{\kappa}})_{\hat{\kappa}\in S}$. It yields a new atlas, which is still denoted by $\mathfrak{A}$. An important observation is that Proposition~\ref{S2: Retraction of SB}, \ref{S2: retraction of H-LH} and Lemma~\ref{main lemma 3.1} still hold with the modified atlas. Moreover, we can define functions belonging to $C^{\omega}(\U,V)$ in terms of $\mathfrak{A}$. Note that the topology of $\F^s(\M,V)$ is independent of the choice of the atlas. One can check that all the theorems in Section~3 remain true with respect to the new atlas. 
\end{rmk}
\medskip

\section{\bf The Ricci-DeTurck Flow}

We first look at the Ricci flow formulated by R. Hamilton. The Ricci flow deforms the metric tensor $g$ of a $m$-dimensional compact closed manifold by the law:
\begin{align}
\label{Ricci flow equation 0}
\begin{cases}
\partial_t g=-2\Ric(g),\\
g(0)=g_0,
\end{cases}
\end{align}
where $\Ric(g)$ is the Ricci tensor of the metric $g$ and $g_0$ is a metric. We will treat equation \ref{Ricci flow equation 0} in its modified version formulated in B. Chow and D. Knopf \cite{Chow04}. The authors show therein the equivalence of their formulation for the Ricci-DeTurck flow to the work of D. DeTurck. It will be shown in this section that on a compact $C^{\infty}$-Riemannian manifold $(\M,\tilde{g})$ the solution to the Ricci flow \ref{Ricci flow equation 0} is analytic in time with respect to the $BC(\M,T^{0}_{2}\M)$-topology. As is known, $(\M,\tilde{g})$ admits a compatible real analytic structure. It follows from \cite{Morrey58} that there is a real analytic metric on $\M$. Without loss of generality, we may assume that $\mathfrak{A}$ and $\tilde{g}$ are real analytic.

We will introduce some notations and concepts prior to the analysis of the Ricci flow. For some fixed interval $I=[0,T]$, $\gamma\in(0,1]$ and Banach space $X$, we define
\begin{align*}
&BU\!C_{1-\gamma}(I,X):=\{u\in{C(\dot{I},X)};[t\mapsto{t^{1-\gamma}}u]\in{BU\!C(\dot{I},X)},\lim\limits_{t\to{0^+}}{t^{1-\gamma}}\|u\|=0\},\\
& \|u\|_{C_{1-\gamma}}:=\sup_{t\in{\dot{I}}}{t^{1-\gamma}}\|u(t)\|_{X},
\end{align*}
and
\begin{align*}
BU\!C_{1-\gamma}^1(I,X):=\{u\in{C^1(\dot{I},X)}: u,\dot{u}\in{BU\!C_{1-\gamma}(I,X)}\}.
\end{align*}
In particular, we put
\begin{align*}
BU\!C_0(I,X):=BU\!C(I,X)\hspace*{1em}\text{ and }\hspace*{1em} BU\!C^1_0(I,X):=BU\!C^1(I,X).
\end{align*}
In addition, if $I=[0,T)$ is a half open interval, then
\begin{align*}
&C_{1-\gamma}(I,X):=\{v\in{C(\dot{I},X)}:v\in{BU\!C_{1-\gamma}([0,t],X)},\hspace{.5em} t<T\},\\
&C^1_{1-\gamma}(I,X):=\{v\in{C^1(\dot{I},X)}:v,\dot{v}\in{C_{1-\gamma}(I,X)}\}.
\end{align*}
We equip these two spaces with the natural Fr\'echet topology induced by the topology of $BU\!C_{1-\gamma}([0,t],X)$ and $BU\!C_{1-\gamma}^1([0,t],X)$, respectively. 
\smallskip\\
Let $l\in{\N}_0$ and $\ev:\Gamma(\M,V^{\sigma+\tau+l}_{\tau+\sigma}\times V^0_l) \rightarrow \Hom(V)$ be the complete contraction.
A linear differential operator $\mathcal{A}(\boldsymbol{\mathfrak{a}}):=\sum\limits_{r=0}^l \ev(a^r,\nabla^r \cdot)$ on the compact closed manifold ${\M}$ is called a \emph{normally elliptic operator of order $l$},
if its principal symbol 
\begin{align*}
\hat{\sigma}\mathcal{A}^{\pi}(\p,\xi(\p)):=\ev(a^l,(-i\xi)^{\otimes l})(\p)\in \L(T_{\p}\M^{\otimes\sigma}\otimes T_{\p}^{\ast}\M^{\otimes\tau})
\end{align*}
satisfies that for all $(\p,\xi)\in \M\times\Gamma(\M, V^0_1)$ with $|\xi|_{g^{\ast}}=1_{\M}$
\begin{align*}
\mathbb{C}^+:=\{z\in\mathbb{C}:Re(z)\geq 0\} \subset\rho(-\hat{\sigma}\mathcal{A}^{\pi}(\p,\xi(\p))).
\end{align*}
This condition can be equivalently stated as follows. $\mathcal{A}$ is normally elliptic iff there exist $0<r<R$ such that the spectrum of the principal symbols of its local expressions $\mathcal{A}_{\kappa}(x,\partial)$, i.e., $\hat{\sigma}\mathcal{A}^{\pi}_{\kappa}(x,\xi):=\Sigma_{|\alpha|=l}a_{\alpha}(x)(-i\xi)^{\alpha}\in\L(E)$ with $(x,\xi)\in \Q\times S^{m-1}$, are contained in 
\begin{align*}
\{z\in\mathbb{C}:Re(z)\geq r\}\cap \{z\in\mathbb{C}:|z|\leq R\}, \hspace*{.5em}\text{for all }(x,\xi)\in \Q\times S^{m-1}.
\end{align*}

For notational convenience, we set 
\begin{align*}
E_0:=bc^{\alpha}({\M},V),\hspace{1em} and\hspace{1em} E_1:=bc^{2+\alpha}({\M},V), 
\end{align*}
and
\begin{align*}
S\!E_0:=bc^{\alpha}({\M},SV),\hspace{1em}\text{ and }\hspace{1em} S\!E_1:=bc^{2+\alpha}({\M},SV) 
\end{align*}
for some $0<\alpha<1$ and $V=V^0_2:=\{T^0_2\M, (\cdot|\cdot)_{\tilde{g}}\}$. $SV$ stands for the symmetric sections of the $(0,2)$-tensor bundle, namely that $a=a_{ij}dx^i\otimes dx^j\in\Gamma(\M,SV)$ iff $a_{ij}(\p)=a_{ji}(\p)$ for every $\p\in\M$. Lastly, observe $E=\R^{1\times m^2}=\R^{m^2}$ and 
\begin{align*}
\mho:=\{g\in S\!E_1: g \text{ is positive definite}\}.
\end{align*}
Here $g$ is said to be positive definite if there exists a constant $c>0$ such that 
\begin{align*}
\langle{g(p)X},X\rangle \geq c|X|_{\tilde{g}(p)}^2
\end{align*}
at all $\p \in\M$ and $X\in T_{\p} \M$. Here $\langle{\cdot,\cdot}\rangle:T^{\ast}{\M}\times T{\M} \rightarrow \R^{\M}$ is the fiber-wise defined duality pairing on ${\M}$.
\smallskip\\
In the first place, I will present DeTurck's trick in getting a modified formulation for equation \eqref{Ricci flow equation 0}. For any sufficiently smooth metric $g$, we can define a global vector field $W_g=W^k_g\frac{\partial}{\partial x^k}$ on $\M$ by 
\begin{align*}
W^k_g=g^{pq}(\Gamma^k_{pq}-\tilde{\Gamma}^k_{pq}), 
\end{align*}
where $\tilde{\Gamma}^k_{pq}$ are the Christoffel symbols of the fixed real analytic background metric $\tilde{g}$. Note that $W_g$ is a well-defined global vector field, since the difference of two connections is a $(1,2)$-tensor.
Taking the Lie derivative of the metric tensor $g$ with respect to $W_g$ induces a map $P$: ${\mathcal{ST}}^0_2{\M}\rightarrow {\mathcal{ST}}^0_2{\M}$, namely,
\begin{align*}
P(g)=\L_{W_g} g,
\end{align*}
which is a second order nonlinear differential operator acting on $g$. Here ${\mathcal{ST}}^0_2{\M}$ denotes the smooth sections of $SV$. 
\smallskip\\
For the sake of working with a fixed atlas independent of time, henceforth we will treat equation \eqref{Ricci flow equation 0} in the atlas $\mathfrak{A}$ instead of the commonly used geodesic normal coordinates with respect to the evolving metric $g(t)$. 
\smallskip\\
Let $S\!E:=\{a\in E: a_{ij}=a_{ji}\}$ with the subscripts ordered lexicographically as aforementioned. Then $u\in \Gamma(\M,SV)$ iff ${\kfk}u\in S\!E^{\Q}$ for all $\kappa\in\mathfrak{K}$. We have 
\begin{align}
\label{section 7: invariance of SV under tu}
{\ttm}(\Gamma(\M, SV))=\Gamma(\M, SV).
\end{align}
An immediate observation is that the spaces $\F^s(\M,SV)$ and $\F^s(\R^m,S\!E)$ are closed in $\F^s(\M,V)$ and $\F^s(\R^m,E)$ with $\F\in\{bc,BC\}$, respectively, and thus are Banach spaces. It is also easy to see that Proposition~\ref{S2: retraction of H-LH} still holds true with $V$ and $E$ replaced by $SV$ and $S\!E$, respectively. Hereafter, I will use some of the results established in the previous sections and \cite{ShaoSim13} with the minute modification that $V$ and $E$ are replaced by $SV$ and $S\!E$. One may check that their proofs remain true unless the necessary modifications are pointed out.
\smallskip\\
In the following, we seek a solution to the Ricci-DeTurck flow:
\begin{align}
\label{Ricci-DeTurck flow equation 1}
\begin{cases}
\partial_t g=-2\Ric(g)+P(g):=Q(g),\\
g(0)=g_0,
\end{cases}
\end{align}
where $g_0$ is the initial metric of $\M$ in \eqref{Ricci flow equation 0}. In \cite{DeTur83}, the author shows that for any smooth initial metric $g_0$, the initial-value problem \eqref{Ricci-DeTurck flow equation 1} has a unique smooth solution $g(t)$ existing on $J(g_0):=[0,T(g_0))$ for some $T(g_0)>0$.

We will adopt an implicit function theorem argument below to show analyticity of $g$. To this end, we first use continuous maximal regularity theory established in \cite{ShaoSim13} to give an existence theorem for equation \eqref{Ricci-DeTurck flow equation 1} with $bc^{2+\alpha}$-continuous initial data.
Let $(Q(g))_{ij}$ be the components of $Q(g)$ and $(Q_{\kappa}(g))_{ij}:=({\kfk}Q(g))_{ij}$. In every local chart $(\Uk,\varphi_{\kappa})$, it holds that
\begin{align}
\label{formula of Q-Ricci}
\notag\partial_t g_{ij}&=(Q_{\kappa}(g))_{ij}=-2R_{ij}+P_{ij}\\
\notag&=-2(\partial_k\Gamma^{k}_{ij}-\partial_j\Gamma^{k}_{ik}+\Gamma^{k}_{kl}\Gamma^{l}_{ij}-\Gamma^{k}_{lj}\Gamma^{l}_{ik})\\
\notag&\hspace*{1em}+\nabla_{W_g}g_{ij}-(\partial_i|\nabla_{W_g}\partial_j)_g-(\nabla_{W_g}\partial_i|\partial_j)_g+(\partial_i|\nabla_{j}W_g)_g+(\nabla_{i}W_g|\partial_j)_g\\
\notag&=-2(\partial_k\Gamma^{k}_{ij}-\partial_j\Gamma^{k}_{ik}+\Gamma^{k}_{kl}\Gamma^{l}_{ij}-\Gamma^{k}_{lj}\Gamma^{l}_{ik})
    +W^k_g\partial_{k}g_{ij}+g_{kj}\partial_{i}W^k_g+g_{ik}\partial_{j}W^k_g\\
&=g^{kl}\partial_{kl}g_{ij}+\mathcal{S}_{\kappa, ij}(g).
\end{align}
$[g\mapsto \mathcal{S}_{\kappa,ij}(g)]\in C^{\omega}(bc^{2+\alpha}(\Q,S\!E),bc^{1+\alpha}(\Q))\cap C^{\omega}(bc^{1+\alpha}(\Q,S\!E),bc^{\alpha}(\Q))$ is a rational function involving the components of $g$ and their first order derivatives with $BC^{\infty}\cap C^{\omega}$-coefficients. Here $\nabla_{i}:=\nabla_{\partial_{i}}$ and $\nabla_{W_g}=W^i_g \nabla_i$. The above computation shows that $[g\mapsto Q_{\kappa}(g)]\in C^{\omega}(bc^{2+\alpha}(\Q,S\!E),bc^{\alpha}(\Q,S\!E))$. We attain
\begin{align*}
Q \in C^{\omega}(\mho, S\!E_0).
\end{align*}
The symbol of the principal part of $(-DQ(g))_{\kappa}$, the local expression for the Fr\'echet derivative of $-Q(g)$ in the local patch $(\Uk,\varphi_{\kappa})$, equals 
\begin{align*}
\hat{\sigma}(-DQ(g))_{\kappa}^{\pi}(x,\xi)=g^{kl}(x)\xi_k \xi_l I_{\R^{m^2}},\hspace*{.5em} (x,\xi)\in \Q\times S^{m-1}.
\end{align*}
The positive definiteness of $g$ immediately implies that $-DQ(g)$ is normally elliptic for each $g\in \mho$. Because $DQ(g) \in \L(S\!E_1,S\!E_0)$, an inspection into the proof of \cite[Theorem~3.2]{ShaoSim13} reveals that it still works for the spaces $S\!E_0$ and $S\!E_1$. Thus by \cite[Theorem~3.4]{ShaoSim13}, we conclude that for each $g\in \mho$,
\begin{align*}
-DQ(g) \in \mathcal{H}(S\!E_1,S\!E_0). 
\end{align*}
It follows from \cite[Theorem~3.6]{ShaoSim13} that 
\begin{align}
\label{MR of RF}
-DQ(g) \in \mathcal{M}_{1}(S\!E_1,S\!E_0), \hspace*{1em} g\in \mho.
\end{align}
\cite[Theorem~2.7]{Ange90} and a prolongation of the solution as in \cite[Theorem~4.1]{CleSim01} imply the following existence theorem.
\begin{theorem}
\label{existence for RF-Ange90}
Suppose that $g_0\in \mho:=\{g\in bc^{2+\alpha}(\M,SV): g \text{ positive definite}\}$ with some fixed $\alpha\in (0,1)$. Then equation \eqref{Ricci-DeTurck flow equation 1} has a unique maximal solution 
\begin{align*}
\hat{g}\in C^{1}(J(g_0),bc^{\alpha}(\M,SV)) \cap C(J(g_0),\mho)
\end{align*}
on the maximal interval of existence $J(g_0):=[0,T(g_0))$ with some $T(g_0)>0$. 
\end{theorem}
\smallskip
Henceforth, $\hat{g}$ is used exclusively for the solution to equation \eqref{Ricci-DeTurck flow equation 1} obtained in Theorem \ref{existence for RF-Ange90}. Given any $(t_0,\p)\in\dot{J}(\rho_0)\times{\M}$, we introduce a family of parameter-dependent diffeomorphisms $\{{\ttl}:(\lambda,\mu)\in\B\}$ around $(t_0,\p)$ for sufficiently small $r>0$. Let $J:=[0,T]\subset J(\rho_0)$ with $t_0\in \mathring{J}$. Define 
\begin{align*}
{\ez}(J):=C(J,E_0), \hspace*{1em} {\ef}(J):=C(J,E_1)\cap{C^1(J,E_0)},
\end{align*}
and
\begin{align*}
{\sez}(J):=C(J,S\!E_0), \hspace*{1em} {\sef}(J):=C(J,S\!E_1)\cap{C^1(J,S\!E_0)},
\end{align*}
and set 
\begin{align*}
{\Ub}(J):=\{u\in\sef(J): u(t)\in \mho,\text{ for all } t\in J\}.
\end{align*}
Clearly, ${\Ub}(J)$ is open in $\sef(J)$. Let $u:=\hat{g}_{\lambda,\mu}=\ttl{\hat{g}}$. Then by Proposition~\ref {transformed functions}, $u$ satisfies the equation
\begin{align*}
\label{transformed RDF}
\partial_t u&=(1+\xi^{\prime}\lambda){\ttl}\hat{g}_t+{B}_{\lambda,\mu}(u)\\
&=(1+\xi^{\prime}\lambda){\ttl}Q(\hat{g})+{B}_{\lambda,\mu}(u)\\
&=(1+\xi^{\prime}\lambda){\tu}Q({\rh}\hat{g})+{B}_{\lambda,\mu}(u)\\
&=(1+\xi^{\prime}\lambda){\tu}Q({\tui}u)+{B}_{\lambda,\mu}(u):=-H_{\lambda,\mu}(u).
\end{align*}
We define a map $\Phi:{\Ub}(J) \times \B \rightarrow {\sez}(J)\times S\!E_1$ by:
\begin{align}
\Phi(g,(\lambda,\mu))=\dbinom{\partial_t g+H_{\lambda,\mu}(g)}{\gamma_{0}(g)-\hat{g}(0)}
\end{align}
with $\gamma_{0}$ standing for the evaluation map at $t=0$, i.e., $\gamma_{0}(g)=g(0)$. The subsequent step is to verify the conditions of the implicit function theorem.
\smallskip\\
(i) Note that $\Phi(\hat{g}_{\lambda,\mu},(\lambda,\mu))=(0,0)^T$ for any $(\lambda,\mu)\in{\B}$. One can compute the Fr\'echet derivative of $\Phi$ in the first component, namely, $D_1\Phi$:
\begin{align*}
D_1\Phi(g,(\lambda,\mu))h=\dbinom{\partial_t h-(1+\xi^{\prime}\lambda){\tu}DQ({\tui}g){\tui}h-{B}_{\lambda,\mu}(h)}{\gamma_{0}h}.
\end{align*}
Thus it yields
\begin{align*}
D_1\Phi(\hat{g},(0,0))h=\dbinom{\partial_t h-DQ(\hat{g})h}{\gamma_{0} h}.
\end{align*}
By \eqref{MR of RF}, for every $t\in J$, 
\begin{align}
\label{S4: ptwise MR}
D_1\Phi(\hat{g}(t),(0,0))=(\frac{d}{dt}-DQ(\hat{g}(t)),\gamma_{0})^T\in{\Lis}({\sef}(J),{\sez}(J)\times{S\!E_1}).
\end{align}
\begin{lem}
\label{section 7: MR of linearization}
$D_1\Phi(\hat{g},(0,0))=(\frac{d}{dt}-DQ(\hat{g}(\cdot)),\gamma_{0})^T\in{\Lis}({\sef}(J),{\sez}(J)\times{S\!E_1})$.
\end{lem}
\begin{proof}
The statement follows from \eqref{S4: ptwise MR} and \cite[Lemma~2.8(a)]{CleSim01}.
\end{proof}
(ii) Now it remains to show that $\Phi\in C^{\omega}({\Ub}(J)\times\B,\sez(J)\times S\!E_1)$. 
\smallskip\\
By Proposition~\ref {transformed functions} and \eqref{section 7: invariance of SV under tu}, $[(\lambda,\mu)\mapsto {B}_{\lambda,\mu}]\in{C^{\omega}({\B},C(J,\mathcal{L}(S\!E_1,S\!E_0)))}$. We define a bilinear and continuous map $f$ by:
\begin{align*}
f:C(J,\mathcal{L}(S\!E_1,S\!E_0))\times {\sef}(J) \rightarrow{\sez}(J),\hspace{.5em} (T(t),g(t))\mapsto T(t)g(t).
\end{align*}
Since $f$ is real analytic, it yields 
\begin{align*}
[(g,(\lambda,\mu))\mapsto{f({B}_{\lambda,\mu},g)}={B}_{\lambda,\mu}(g)]\in{C^{\omega}({\Ub}(J)\times{\B},{\sez}(J))}.
\end{align*}
From equation \eqref{formula of Q-Ricci}, we know that, in every local chart $(\Uk,\varphi_{\kappa})$ and for any $1\leq i,j \leq m$, $Q(g)$ can be written in the form of 
\begin{align*}
(Q_{\kappa}(g))_{ij}=\frac{\sum_{k}{A}^{k_1}_{\kappa,ij}g \cdots {A}^{k_r}_{\kappa,ij}g}{\det[g]^3},
\end{align*}
where the ${A}^{r}_{\kappa,ij}$ are linear differential operators of order at most two belonging to $\L(bc^{2+\alpha}(\Q,E),bc^{\alpha}(\Q))$ acting on the components of $g$ whose coefficients are in $ BC^{\infty}(\Q,\L(E,\R))\cap C^{\omega}(\Q,\L(E,\R))$. $\det[g]$ denotes the determinant of the first fundamental form for the metric $g$. The power $3$ of $\det[g]$ follows from the well-known formula for inverse matrix.
\smallskip\\
We set $\tilde{\pi}_{\kp}:=\zeta_{\kp}+\zeta_{\iota}-\zeta_{\kp}\zeta_{\iota}$ and $\tilde{\pi}_{\kappa}:=\pi_{\kappa}^2-\tilde{\pi}_{\kp}\pi_{\kappa}^2$. 
In virtue of (L2), one can check that $(\tilde{\pi}_{\kappa})_{\kappa\in\mathfrak{K}}$ forms a partition of unity subordinate to $\mathfrak{A}$. Decompose $Q$ into 
\begin{align*}
Q(g)=\sum_{\kappa}\tilde{\pi}_{\kappa} Q(g).
\end{align*}
Let $\pi:={\kfp}\tilde{\pi}_{\kp}$. The local expression of $\tilde{\pi}_{\kp} Q(g)$ in $(\Ukp,\varphi_{\kp})$ is of the form 
\begin{align*}
({\kfp}\tilde{\pi}_{\kp} Q_{\kp}(g))_{ij}=\pi \frac{\sum_{k}{A}^{k_1}_{\kp,ij}g \cdots {A}^{k_r}_{\kp,ij}g}{\det[g]^3}.
\end{align*}
We introduce an auxiliary function $\varpi\in\mathcal{D}({\Q},[0,1])$ satisfying $\varpi|_{\supp(\pi)}\equiv 1$ with $\varpi_{\kappa}:={\kbk}\varpi$. We define linear differential operators $\mathcal{A}^s_{ij}:E_1 \rightarrow E_0$ by
\begin{align*}
\mathcal{A}^s_{\kp,ij}g:=({\kbp}\varpi{A}^{s}_{\kp,ij}{\kfp}g) dx^{i}\otimes dx^{j}.
\end{align*}  
These are well-defined operators of order at most two on $\M$. One checks that the coefficients of their local expressions in $(\Ukp,\varphi_{\kp})$ belong to 
\begin{align*}
C^{\omega}(\varphi_{\kp}(\U),\L(E))\cap BC^{\infty}(\Q,\L(E)),
\end{align*}
and are supported in $\supp(\varpi)$. 
By Proposition~\ref{differentiability involving time}(b), we thus have
\begin{align}
\label{S5: G&F}
[\mu\mapsto  {\tu}\mathcal{A}^s_{\kp,ij}{\tui}]\in C^{\omega}(\B, C(J,\L(E_1,E_0))).
\end{align}
Next we define a tensor field $S_{\varpi}\in C^{\infty}(\M,V^{\sigma+2\tau}_{\tau+2\sigma})$ compactly supported in $\Ukp$ by
\begin{align*}
S_{\varpi}:=\varpi_{\kp}\frac{\partial}{\partial x^{(i)}}\otimes \frac{\partial}{\partial x^{(j)}}\otimes \frac{\partial}{\partial x^{(j)}} \otimes dx^{(j)}\otimes dx^{(i)}\otimes dx^{(i)}
\end{align*}
for all $(i)\in\mathbb{J}^{\sigma}, (j)\in\mathbb{J}^{\tau}$ with $\sigma,\tau\in\N_0$. Let $\mathsf{C}$ be the complete contraction. The operation $\mathsf{C}^{\sigma}_{\tau}: \Gamma(\M, V^{\sigma}_{\tau}) \times \Gamma(\M, V^{\sigma}_{\tau}) \rightarrow \Gamma(\M, V^{\sigma}_{\tau})$ is defined by
\begin{align*}
\mathsf{C}^{\sigma}_{\tau}(a,b):=\mathsf{C}(S_{\varpi}, a\otimes b)=\varpi_{\kp} a^{(i)}_{(j)}b^{(i)}_{(j)} \frac{\partial}{\partial x^{(i)}}\otimes dx^{(j)}.
\end{align*}
\begin{lem}
\label{pointwise multiplication properties 3}
The point-wise extension $\mathsf{m}$ of $\mathsf{C}^{\sigma}_{\tau}$ satisfies 
\begin{align*}
[(v,u)\mapsto \mathsf{m}(v,u)]\in \L(bc^s({\M},V),bc^s({\M},V);bc^s({\M},V)).
\end{align*}
\end{lem}
\begin{proof}
By \cite[Examples~13.4(b), Lemma~14.2]{AmaAr}, $\mathsf{C}^{\sigma}_{\tau}$ is a bundle multiplication. The assertion hence follows by Proposition~\ref{pointwise multiplication properties}.
\end{proof}
Using the notation in the lemma above, we can write
\begin{align*}
\tilde{\pi}_{\kp}{\kbp}(\tilde{Q}_{\kp}(g))_{ij}dx^{i}\otimes dx^{j}=\frac{\tilde{\pi}_{\kp}}{\mathcal{P}_{\kp}(g)}\sum_k \mathsf{m}(\cdots (\mathsf{m}( \mathcal{A}^{k_1}_{\kp,ij}g, \mathcal{A}^{k_2}_{\kp,ij}g),\cdot)\cdots,\mathcal{A}^{k_r}_{\kp,ij}g).
\end{align*}
Here $\mathcal{P}_{\kp}(g):={\kbp}(\varpi\det[g]^3)+(1_{\M}-\varpi_{\kp})>0$. 
It follows from an analogous argument as above, \cite[Proposition~6.4]{ShaoSim13} and Proposition~\ref{pointwise multiplication properties}, \ref{differentiability involving time}(a) that
\begin{align}
\label{section 7 eq1}
[(g,\mu) \mapsto{\tu}\frac{\tilde{\pi}_{\kp}}{\mathcal{P}_{\kp}({\tui} g)}]\in C^{\omega}({\Ub}(J)\times\B,C(J,bc^{2+\alpha}(\M))).
\end{align}
By Proposition~\ref{pointwise multiplication properties}, Lemma \ref{pointwise multiplication properties 3}, \eqref{S5: G&F} and \eqref{section 7 eq1}, we hence conclude that 
\begin{align*}
[(g,\mu)\mapsto {\tu}\{\tilde{\pi}_{\kp}{\kbp}(\tilde{Q}_{\kp}({\tui}g))_{ij} dx^{i}\otimes dx^{j}\}]\in C^{\omega}({\Ub}(J)\times\B,\ez(J)).
\end{align*}
Since $\tilde{\pi}_{\kp} Q(g)=\sum_{i,j=1}^{m}\tilde{\pi}_{\kp}{\kbp}(\tilde{Q}_{\kp}(g))_{ij} dx^{i}\otimes dx^{j}$, it yields
\begin{align*}
[(g,(\lambda,\mu))\mapsto(1+\xi^{\prime}\lambda){\tu}\tilde{\pi}_{\kp} Q({\tui}g)]\in C^{\omega}({\Ub}(J)\times\B,\ez(J)).
\end{align*}
Applying this argument to all other components $\tilde{\pi}_{\kappa} Q$, by \eqref{section 7: invariance of SV under tu} we immediately have
\begin{align*}
[(g,(\lambda,\mu))\mapsto(1+\xi^{\prime}\lambda){\tu} Q({\tui}g)]\in C^{\omega}({\Ub}(J)\times\B,\sez(J)).
\end{align*}

The implicit function theorem implies that there is a $\mathbb{B}(0,r_0)\subset{\B}$ such that
\begin{align*}
[(\lambda,\mu)\mapsto{\hat{g}}_{\lambda,\mu}]\in{C^{\omega}(\mathbb{B}(0,r_0),{\sef}(J))}.
\end{align*}
Because $(t_0,\p)$ is arbitrary and $E_1\hookrightarrow{BC({\M},V)}$, it follows from Theorem \ref{S1: Theorem} that
\begin{theorem}
\label{Analyticity of Ricci-DeTurck}
The solution $\hat{g}$ in Theorem~\ref{existence for RF-Ange90} satisfies $\hat{g}\in C^{\omega}(\dot{J}(g_0)\times\M,SV$).
\end{theorem}

The one-parameter family of vector fields $W_{\hat{g}}(t)$ exists as long as the solution $\hat{g}(t)$ to \eqref{Ricci-DeTurck flow equation 1} exists, i.e., $W_{\hat{g}}(t)$ lives on $J(g_0)$. As an immediate consequence of Theorem \ref{Analyticity of Ricci-DeTurck}, we attain
\begin{align}
\label{S4: W}
W_{\hat{g}}\in C^{\omega}(\dot{J}(g_0)\times\M,T\M).
\end{align} 
One defines a one-parameter family of maps $\phi_t:\M \rightarrow \M$ by
\begin{align}
\label{manifold diffeomorphism}
\begin{cases}
\partial_t \phi_t(\p)=-W_{\hat{g}}(t,\phi_t(\p)),\\
\phi_0={\id}_{\M}.
\end{cases}
\end{align}
The proof of \cite[Lemma~3.15]{Chow04} shows that the solution to this system of ODE exists smoothly on $\dot{J}(g_0)$ and remains diffeomorphisms for all time.
\begin{prop}\cite[Section~3.3: Step 3]{Chow04}
\label{Ricci-DeTurck to Ricci}
The family of metrics $\bar{g}(t):=\phi^{\ast}_t \hat{g}(t)$ with $t\in J(g_0)$ is a solution to the Ricci flow \eqref{Ricci flow equation 0}.
\end{prop}
\begin{proof}
We will present a brief proof for the reader's convenience. First observe that $\bar{g}(0)=\hat{g}(0)=g_0$. Then we compute 
\begin{align}
\label{S4: last sol}
\notag\partial_t \phi^{\ast}_t \hat{g}(t)&=\partial_s|_{s=0} \phi^{\ast}_{t+s} \hat{g}(t+s)
=\phi^{\ast}_t[\partial_t \hat{g}(t)]+\partial_s|_{s=0}  \phi^{\ast}_{t+s} \hat{g}(t)\\
\notag&=\phi^{\ast}_t[-2\Ric(\hat{g}(t))+\L_{W_{\hat{g}}(t,\cdot)}\hat{g}(t)]+\partial_s|_{s=0}  [(\phi^{-1}_{t}\circ\phi_{t+s})^{\ast}\phi^{\ast}_{t} \hat{g}(t)]\\
&=-2\Ric(\phi^{\ast}_t \hat{g}(t))+\phi^{\ast}_t(\L_{W_{\hat{g}}(t,\cdot)}\hat{g}(t))-\L_{\phi^{\ast}_{t}W_{\hat{g}}(t,\cdot)}\phi^{\ast}_{t}\hat{g}(t)\\
\notag&=-2\Ric(\phi^{\ast}_t \hat{g}(t)).
\end{align}
\eqref{S4: last sol} follows from the identity:
\begin{align*}
\partial_s|_{s=0}(\phi^{-1}_{t}\circ\phi_{t+s})&=d(\phi^{-1}_{t})\partial_s|_{s=0}\phi_{t+s}
=-d(\phi^{-1}_{t})W_{\hat{g}}(t,\phi_t)
=-\phi^{\ast}_t W_{\hat{g}}(t,\cdot).
\end{align*}
Hence $\bar{g}(t)$ solves the Ricci flow \eqref{Ricci flow equation 0}.
\end{proof}
We prove a theorem on time-analyticity of solutions to the Ricci flow:
\begin{theorem}
\label{Analyticity of Ricci flow}
Suppose that $g_0\in \mho:=\{g\in bc^{2+\alpha}(\M,SV): g \text{ positive definite}\}$ with some fixed $0<\alpha<1$. Then there exists a $T(g_0)>0$ such that on the maximal interval of existence $J(g_0):=[0,T(g_0))$ equation \eqref{Ricci flow equation 0} has a unique local solution 
\begin{align*}
\bar{g}\in C^{\omega}(\dot{J}(g_0),BC(\M, SV)).
\end{align*} 
\end{theorem}
\begin{proof}
Without loss of generality, we assume that $T(g_0)$ is small enough such that for each local patch $(\Uk,\varphi_{\kappa})$ and each $\p\in \psi_{\kappa}(\frac{2+r_0}{3}\overline{\Q})$, $\phi_t(\p)$ remains in $\Uk$, where $r_0\in (0,1)$ is a shrinking constant in the uniformly shrinkable property of $\mathfrak{A}$. Put 
\begin{align*}
h_{\kappa}(t):=\varphi_{\kappa}\circ\phi_t\circ \psi_{\kappa},\hspace{1em} W_{\kappa}:=\kfk W_{\hat{g}}.
\end{align*}
Let $B:=\frac{1+r_0}{2}\overline{\Q}$ and $\breve{B}:=\frac{2+r_0}{3}\overline{\Q}$.
We first show that for $i=1,\cdots,m$
\begin{align}
\label{S4: reg of h & h_i}
{h}_{\kappa}\in C^{\omega}(\dot{J}, C(B,\Q)),\hspace*{.5em}\text{ and }\hspace*{.5em} \partial_i {h}_{\kappa}\in C^{\omega}(\dot{J}, C(B,\R^m)).
\end{align}
We will split \eqref{manifold diffeomorphism} into infinite dimensional vector-valued ODEs. Actually, if we choose to work with the usual $\R^m$-valued ODEs, then we will only be able to show that the flow through every $\p\in\M$ is analytic, i.e., $\bar{g}(\cdot,\p)\in C^{\omega}(\dot{J},T^{\ast}_{\p}\M \otimes T^{\ast}_{\p}\M)$.
\smallskip\\
(i) We first note that ${h}_{\kappa}$ solves the following ODE
\begin{align}
\label{section 7 last eq}
\begin{cases}
\frac{d}{dt} h_{\kappa}(t)=-{\kfk} W_{\hat{g}}(t,h_{\kappa}(t))=-{W}_{\kappa}(t,h_{\kappa}(t)),\\
h_{\kappa}(0)={\id}_{B}.
\end{cases}
\end{align}
Thanks to $\hat{g}\in  C(J, \mho)\cap C^{\omega}(\dot{J}\times\M,SV)$, \eqref{S3: kfk is bdd} and \eqref{S4: W}, we obtain
\begin{align}
\label{S4: reg of Wk}
{W}_{\kappa}\in  C(J,bc^{1+\alpha}(\Q,\R^m))\cap C^{\omega}(\dot{J}\times \Q,\R^m).
\end{align}
Let $X:=C(B,\Q)$ and $Y:=C(B,\R^m)$. For any $u\in X$ and $v\in Y$ with $\|v\|_Y$ so small that $u+v\in X$, we have
\begin{align*}
&\quad W_{\kappa}(t,u+v)(x)-W_{\kappa}(t,u)(x)-\partial_2 W_{\kappa}(t,u(x))v(x)\\
&=\int\limits_{0}^1 [\partial_2 W_{\kappa}(t,u(x)+sv(x))-\partial_2 W_{\kappa}(t,u(x))]\, ds \; v(x)\\
&=\int\limits_{0}^1 [\partial_2 W_{\kappa}(t,u+sv)-\partial_2 W_{\kappa}(t,u)]\, ds(x) \; v(x)
\end{align*}
for $x\in B$.
It is not hard to check that $W_{\kappa}(t,\cdot)\in BU\!C^1(\Q,\R^m)$. Thus the expression in the last line converges to $0$ in $Y$ as $\|v\|_Y\rightarrow 0$. Together with the second part of \eqref{S4: reg of Wk}, it implies that
\begin{align}
\label{S4: reg of Wk-vec}
W_{\kappa}\in C^{0,1}(J\times X,Y)\cap C^{\omega}(\dot{J}\times X,Y).
\end{align}
By the Picard-Lindel\"of Theorem, there exists a unique solution $\hat{h}_{\kappa}\in C^1(J,X)$ to \eqref{section 7 last eq} on $J_1:=[0,T_1]$. Given any $t_0\in \mathring{J}_1$, we pick $I:=[\varepsilon,T]\subset \mathring{J}_1$ with $t_0\in \mathring{I}$. Define $\Phi: C^1(I,X)\times \B  \rightarrow C(I,Y)\times Y$ by
\begin{align*}
\Phi(u,\lambda)=
\binom{\partial_t u(t)+(1+\xi^{\prime}(t)\lambda) W_{\kappa}(t+\xi(t)\lambda, u(t))}{\gamma_{\varepsilon}u-\hat{h}_{\kappa}(\varepsilon)}.
\end{align*}
Note that $\Phi({\rh}\hat{h}_{\kappa},\lambda)=(0,0)^T$. \eqref{S4: reg of Wk-vec} implies that $D_2 W_{\kappa}(\cdot,\hat{h}_{\kappa}(\cdot))\in C(I,\L(Y))$, where $D_2 W_{\kappa}$ denotes the Fr\'echet derivative of $W_{\kappa}$ with respect to $Y$. By standard ODE theory, we obtain
\begin{align*}
D_1\Phi(\hat{h}_{\kappa},0)\in \Lis(C^1(I,Y),C(I,Y)\times Y).
\end{align*}
On the other hand, we can also infer from \eqref{S4: reg of Wk-vec} that for every $(t_0,x_0)\in \dot{J}\times X$, there exist constants $M,R,r_1$ depending on $(t_0,x_0)$ such that for all $(s,x)\in \mathbb{B}(t_0,r_1)\times\mathbb{B}_{Y}(x_0,r_1)$
\begin{align}
\label{S4: Cauchy est}
\|D^{\beta}W_{\kappa}(s,x)\|\leq M\frac{\beta!}{R^{|\beta|}}\hspace{1em} \text{ with }\beta\in\N_0^2.
\end{align}
Given $u_0\in C(I,X)$, let $\mathcal{R}:={\im}(u_0)\subset X$. Since $I\times\mathcal{R}$ is compact in $\dot{J}\times X$, it follows from a compactness argument that there exist uniform constants $M,R$ such that \eqref{S4: Cauchy est} holds for all $(s,u_0(t))$ with $s,t\in I$. We conclude that for every $\lambda_0$
\begin{align*}
W_{\kappa}(t+\xi(t)\lambda,u)=\sum\limits_{\beta} \frac{1}{\beta!} D^{\beta}W_{\kappa}(t+\xi(t)\lambda_0,u_0(t))((\lambda,u(t))-(\lambda_0,u_0(t)))^{\beta}
\end{align*}
converges in $C(I,Y)$ for sufficiently small $r>0$ and $(u,\lambda)\in \mathbb{B}_{C(I,Y)}(u_0,r)\times\B$. This implies 
\begin{align*}
\Phi\in C^{\omega}(C^1(I,X)\times\B, C(I,Y)\times Y). 
\end{align*}
The implicit function theorem and a similar argument to Theorem~\ref{main theorem-iff} now yield 
\begin{align}
\label{S4: h-reg}
\hat{h}_{\kappa}\in C^{\omega}(\mathring{J}_1,X). 
\end{align}
Likewise, we also consider $\hat{h}_{\kappa}$ to be define on $\breve{B}$. Then $\hat{h}_{\kappa}\in C^{\omega}(\mathring{J}_1,C(\breve{B},\Q))$.
\smallskip\\
(ii) $\hat{h}_{\kappa}$ satisfies the following equation for every $x\in \breve{B}$ on $J_1$: 
\begin{align}
\label{S4: h-ptwise-eq}
\begin{cases}
\frac{d}{dt} {h}_{\kappa}(t)(x)=-{W}_{\kappa}(t,{h}_{\kappa}(t)(x)),\\
h_{\kappa}(0)(x)=x.
\end{cases}
\end{align}
\eqref{S4: reg of Wk} yields $W_{\kappa}\in C^{0,1}(J\times\Q,\R^m)$. By \cite[Theorem~9.2]{Ama90}, $\hat{h}_{\kappa}(t):B\rightarrow \Q$: $x\mapsto \hat{h}_{\kappa}(t)(x)$ is continuously differentiable with respect to $x$ for every $t\in J_1$. Let $A(t):=-D_2 W_{\kappa}(t,\hat{h}_{\kappa}(t))$. Consider the following ODE on $I$:
\begin{align}
\label{S4: part der}
\begin{cases}
\frac{d}{dt} v(t)=A(t)v(t),\\
v(\varepsilon)=\partial_i \hat{h}_{\kappa}(\varepsilon).
\end{cases}
\end{align}
Since $A \in C(J_1,\L(Y))$ and $\partial_i \hat{h}_{\kappa}(\varepsilon)\in Y$, by \cite[Remarks~7.10(c)]{Ama90} there is a unique global solution $\hat{v}\in C^1(I,Y)$. In particular, ${\rh}\hat{v}$ satisfies
\begin{align*}
\partial_t ({\rh}\hat{v})=(1+\xi^{\prime}(t)\lambda)A(t+\xi(t)\lambda){\rh}\hat{v}(t)=(1+\xi^{\prime}(t)\lambda){\rh}A(t){\rh}\hat{v}(t).
\end{align*}
\eqref{S4: reg of Wk} and \eqref{S4: h-reg} imply that $A \in C^{\omega}(\mathring{J}_1,\L(Y))$. Therefore, it follows that
\begin{align*}
[\lambda\mapsto {\rh}A]\in C^\omega(\B, C(I,\L(Y))).
\end{align*}
By \cite[Remarks~7.10(c)]{Ama90} and an implicit function theorem argument similar to step (i), we can thus infer that 
\begin{align*}
[t\mapsto \hat{v}(t)]\in C^{\omega}(\mathring{J}_1,Y).
\end{align*}
An easy computation shows that $\partial_i \hat{h}_{\kappa}$ satisfies \eqref{S4: part der} point-wise on $B$, i.e., it solves 
\begin{align*}
\begin{cases}
\partial_t v(t,x)=-D_2 W_{\kappa}(t,\hat{h}_{\kappa}(t)(x))v(t,x),\\
v(0,x)=e_i
\end{cases}
\end{align*}
for all $x\in B$. By uniqueness of the solution to the above ODE, we infer that 
\begin{align*}
\partial_i \hat{h}_{\kappa}=\hat{v}\in C^{\omega}(\mathring{J}_1,Y),\hspace{1em}i=1,\cdots,m.
\end{align*}
We obtain \eqref{S4: reg of h & h_i} by arguing similarly for all continuations of $\hat{h}_{\kappa}$ in all local patches.
\smallskip\\
(iii) Theorem~\ref{Analyticity of Ricci-DeTurck} implies that ${\kfk}\hat{g}\in C^{\omega}(\dot{J}\times\Q,S\!E)$, which in turn yields
\begin{align*}
{\kfk}\hat{g}\in C^{\omega}(\dot{J}\times C(B,\Q),C(B,S\!E)),\hspace{1em}\kappa\in\mathfrak{K}.
\end{align*}
Then we attain
\begin{align*}
({\kfk}\phi_t^{\ast}\hat{g})_{ij}(t,x)=({\kfk}\hat{g})_{kl}(t,\hat{h}_{\kappa}(t)(x))(\partial_i \hat{h}_{\kappa}(t)(x))^k (\partial_j \hat{h}_{\kappa}(t)(x))^l.
\end{align*}   
Here $(\partial_i \hat{h}_{\kappa})^k$ denotes the $k$-th entry of $\partial_i \hat{h}_{\kappa}$, and likewise for $(\partial_j \hat{h}_{\kappa})^l$. \eqref{S4: reg of h & h_i} implies 
\begin{align*}
({\kfk}\phi_t^{\ast}\hat{g})_{ij}\in C^{\omega}(\dot{J}, C(B)).
\end{align*}
Without loss of generality, we may assume that $(\pi_{\kappa})_{\kappa}$ is subordinate to the open cover $(\psi_{\kappa}(\mathring{B}))_{\kappa}$ of $\M$, see \cite[Lemma~3.2]{Ama13} for justification. Then we obtain
\begin{align*}
{\kfk}\pi_{\kappa}\phi_t^{\ast}\hat{g}={\kfk}\pi_{\kappa}{\kfk}\phi_t^{\ast}\hat{g}\in C^{\omega}(\dot{J}, BC(\R^m,S\!E)).
\end{align*}
Now the statement follows from Proposition~\ref{S2: retraction of H-LH}.
\end{proof} 
\begin{remark}
The result in Theorem~\ref{Analyticity of Ricci-DeTurck} can be extended to non-compact $C^{\omega}$-uniformly regular Riemannian manifolds, as long as $g_0$ is $bc^{2+\alpha}$-continuous for some $\alpha\in (0,1)$ and $g_0\sim \tilde{g}$.
\end{remark}
\medskip

\section{\bf The Surface Diffusion Flow}

The problem of the surface diffusion flow aims at finding a family of smooth immersed oriented hypersurfaces $\Gamma=\{\Gamma(t): t\geq{0}\}$ in ${\R}^{m+1}$ satisfying the equation:

\begin{equation}
\label{SDF equation 1}
\begin{cases}
V(t)=-\Delta_{\Gamma(t)}H_{\Gamma(t)},\\
\Gamma(0)=\Gamma_0{.}
\end{cases}
\end{equation}
Here $V(t)$ denotes the velocity in the normal direction of $\Gamma$ at time $t$ and $H_{\Gamma(t)}$ stands for the mean curvature of $\Gamma(t)$. $\Delta_{\Gamma(t)}$ is the Laplace-Beltrami operator on $\Gamma(t)$.  We choose the orientation induced by the outer normal so that $V(t)$ is positive when the enclosed region is growing and $H_{\Gamma(t)}$ is negative while the enclosed region is convex. 

If we start from an compact closed embedded initial hypersurface $\Gamma_0$ belonging to the class $bc^{s}$ for some $s>2$, then by the discussion in \cite[Section~4]{PruSim13} we can find a $m$-dimensional real analytic compact closed embedded oriented hypersurface $(\M,g)$ with $g$ as the Euclidean metric on $\M$, a function $\rho_0 \in bc^s(\M) $ and a parametrization 
\begin{align*}
\Psi_{\rho_0}:{\M}\rightarrow{\R}^{m+1},\hspace{.5em} \Psi_{\rho_0}(\p):=\p+\rho_0(\p){\nu}_{\M}(\p)
\end{align*}
such that $\Gamma_0={\im}(\Psi_{\rho_0})$. Here ${\nu}_{\M}(\p)$ denotes the unit normal with respect to a chosen orientation of $\M$ at $\p$, and $\rho_0:{\M}\rightarrow (-a,a)$ is a real-valued function on ${\M}$, where $a$ is a sufficiently small positive number depending on the inner and outer ball condition of ${\M}$. The reader may consult \cite[Section~4.1]{PruSim13} for the precise bound of $a$. Thus $\Gamma_0$ lies in the $a$-tubular neighborhood of ${\M}$. In fact, it will suffice to assume $\Gamma_0$ to be a $C^2$-manifold for the existence of such a parametrization and a real analytic reference manifold. See \cite[Section~4]{PruSim13} for a detailed proof.
\smallskip\\
Analogously, if $\Gamma(t)$ is $C^1$-close enough to ${\M}$, then we can find a function $\rho:[0,T)\times{\M}\rightarrow(-a,a)$ for some $T>0$ and a parametrization
\begin{align}
\label{S4: Ref-mfd & para}
\Psi_{\rho}:{[0,T)}\times{\M}\rightarrow{\R}^{m+1},\hspace{.5em} \Psi_{\rho}(t,\p):=\p+\rho(t,\p){\nu}_{\M}(\p)
\end{align}
such that $\Gamma(t)={\im}(\Psi_{\rho}(t,\cdot))$ for every $t\in [0,T)$. 
\smallskip\\
For any fixed $t$, we do not distinguish between $\rho(t,\cdot)$ and $\rho(t,\psi_{\kappa}(\cdot))$ in each local coordinate $({\Uk},\varphi_{\kappa})$ and abbreviate $\Psi_{\rho}(t,\cdot)$ to be $\Psi_{\rho}:=\Psi_{\rho}(t,\cdot)$. In addition, the hypersurface $\Gamma(t)$ will be simply written as $\Gamma_{\rho}$ as long as the choice of $t$ is of no importance in the context, or $\rho$ is independent of $t$. 
\smallskip\\
Let $0<\alpha<1$. We define 
\begin{align*}
E_0:=bc^{\alpha}({\M}),\hspace{1em} E_1:=bc^{4+\alpha}({\M}) \hspace{.5em}\text{ and }\hspace{.5em}E_{\frac{1}{2}}:=(E_0,E_1)^0_{\frac{1}{2},\infty}.
\end{align*}
By Proposition~\ref{interpolation result}, $E_{\frac{1}{2}}=bc^{2+\alpha}(\M)$. Put 
\begin{align*}
\mho:=\{\rho\in E_{\frac{1}{2}}: \|\rho\|_{\infty}^{\M}<a\}. 
\end{align*}
For any $\rho\in \mho$, ${\im}(\Psi_{\rho})$ constitutes a $bc^{2+\alpha}$-hypersurface $\Gamma_{\rho}$. In this case, $\Psi_{\rho}$ defines a $bc^{2+\alpha}$-diffeomorphism from ${\M}$ onto $\Gamma_{\rho}$. 
\smallskip\\
Here and in the following, it is understood that the Einstein summation convention is employed and all the summations run from $1$ to $m$ for all repeated indices. 
\smallskip\\
For simplification, we set $H_{\rho}:=\Psi_{\rho}^{\ast}H_{\Gamma_{\rho}}$. Besides, we have $\Psi_{\rho}^{\ast}\Delta_{\Gamma_{\rho}}=\Delta_{\rho}\Psi_{\rho}^{\ast}$, where $\Delta_{\Gamma_{\rho}}$ and $\Delta_{\rho}$ are the Laplace-Beltrami operators on $(\Gamma_{\rho},g_{\Gamma})$ and $({\M},\sigma(\rho))$. Here $g_{\Gamma}$ is the Euclidean metric on $\Gamma_{\rho}$, and $\sigma(\rho):=\Psi_{\rho}^{\ast}g_{\rho}$ with $\Psi_{\rho}^{\ast}g_{\Gamma}$ standing for the pull-back metric of $g_{\Gamma}$ on ${\M}$. Then
\begin{align*}
\Delta_{\rho}=\sigma^{jk}(\rho)(\partial_j\partial_k-\gamma^i_{jk}(\rho)\partial_i).
\end{align*}
Here $\sigma^{jk}(\rho)$ are the components of the induced metric $\sigma^{\ast}(\rho)$ on the cotangent bundle. Note that $\sigma^{jk}(\rho)$ involves the derivatives of $\rho$ merely up to order one. $\gamma^i_{jk}(\rho)$ are the corresponding Christoffel symbols of $\sigma(\rho)$, which contain the derivatives of $\rho$ up to second order.
\smallskip\\
In \cite{Shao13}, the author derives an expression for $H_{\rho}=P_1(\rho)\rho+F_1(\rho)$:
\begin{align*}
F_1(\rho)=\frac{\beta(\rho)}{2} {g^{ij}_{\Gamma}}({l_{ij}}-\rho{l_{ik}}{l^{k}_{j}})=\frac{\beta(\rho)}{2} {\rm{Tr}}\{[g_{\Gamma}]^{-1}(L^{\M}-\rho{L^{\M}}L_{\M})\},
\end{align*}
where ${\rm{Tr}}(\cdot)$ denotes the trace operator, 
and
\begin{align*}
P_1(\rho)&=\frac{\beta(\rho)}{2}\{{g^{ij}_{\Gamma}}\partial_{ij}
+{g^{ij}_{\Gamma}}({l^{k}_{j}}\partial_{i}\rho-\Gamma^{k}_{ij})\partial_{k}
\\
&\hspace*{1em}+{g^{ij}_{\Gamma}}[r_k^l(\rho)l^{k}_{i}\partial_{j}\rho +
r_k^l(\rho)(\partial_{j}l_{i}^{k}+\Gamma^{k}_{jh}l_{i}^{h}-\Gamma^{h}_{ij}l_{h}^{k})\rho 
+r_k^l(\rho)l^{h}_{j}l^{k}_{h}\rho\partial_{i}\rho ]\partial_{l} \}
\end{align*}
in every local chart. Here $g^{ij}_{\Gamma}$ are the components of the induced metric by $g_{\Gamma}$ on the cotangent bundle, and 
\begin{align}
\label{S5: r^j_i}
r_i^j(\rho)=\frac{P_{i}^{j}(\rho)}{Q_{i}^{j}(\rho)}
\end{align}
in every local chart, where $P_{i}^{j}$ and $Q_{i}^{j}$ are polynomials in $\rho$ with $BC^{\infty}\cap C^{\omega}$-coefficients and $Q_{i}^{j}\neq{0}$. Meanwhile, 
\begin{align}
\label{S5: beta}
\beta_{\kappa}(\rho):={\kfk}\beta(\rho)=[1+g^{ik} r_i^j(\rho) r_k^l(\rho) \partial_{j}\rho \partial_{l}\rho]^{-1/2}. 
\end{align}
In particular, $\beta(\rho)\leq 1$ for any $\rho\in  \mho$. Note that in every local chart 
\begin{align*}
\beta_{\kappa}^2(\rho)=\frac{P^{\beta}(\rho)}{Q^{\beta}(\rho)}, 
\end{align*}
where $P^{\beta}(\rho)$ is a polynomial in $\rho$ with $BC^{\infty}\cap C^{\omega}$-coefficients and $Q^{\beta}(\rho)\neq{0}$ is a polynomial in $\rho$ and $\partial_j \rho$ with $BC^{\infty}\cap C^{\omega}$-coefficients. $\Gamma^{k}_{ij}$ are the Christoffel symbols of the metric $g$. $l^i_j$ and $l_{ij}$ are the components of the Weingarten tensor $L_{\M}$ and the second fundamental form $L^{\M}$ of ${\M}$ with respect to $g$, respectively. In particular, we have
\begin{equation}
\label{S5: gamma_ij}
g^{\Gamma}_{ij}=g_{ij}-2\rho{l}_{ij}+\rho^{2}l^{r}_{i}l_{jr}+\partial_{i}\rho\partial_{j}\rho.
\end{equation} 
The reader may refer to \cite{EscMaySim98, EscSim9802} for a different analysis of the mean curvature operator. Now the first line of equation \eqref{SDF equation 1} is equivalent to
\begin{align*}
\rho_{t}=-\frac{1}{\beta(\rho)}\Psi^{\ast}_{\rho}\Delta_{\Gamma_{\rho}}H_{\Gamma_{\rho}}=-\frac{1}{\beta(\rho)}\Delta_{\rho}H_{\rho}.
\end{align*}
There exists a global operator $R(\rho)\in{\mathcal{L}(bc^{3+\alpha}({\M}),E_0)}$ such that $R$ is well defined on $ \mho$ and:
\begin{align*}
R(\rho)\rho=\frac{1}{2\beta(\rho)}\Delta_{\rho}[\beta(\rho){\rm{Tr}}([g_{\Gamma}]^{-1}L^{\M})]-\frac{\rho}{2\beta(\rho)}\Delta_{\rho}[\beta(\rho){\rm{Tr}}([g_{\Gamma}]^{-1}{L^{\M}}L_{\M})]{.}
\end{align*}
We set
\begin{align*}
&P(\rho):=\frac{1}{\beta(\rho)}\Delta_{\rho}P_1(\rho)+R(\rho),\hspace*{2.4em}\rho\in \mho{,}\\
&F(\rho):=-\frac{1}{\beta(\rho)}\Delta_{\rho}F_1(\rho)+R(\rho)\rho,\hspace*{1em}\rho\in \mho\cap{bc^{3+\alpha}({\M})}.
\end{align*}
Note that third order derivatives of $\rho$ do not appear in $F(\rho)$. Hence it is well-defined on $ \mho$. On account of \eqref{S5: r^j_i}-\eqref{S5: gamma_ij}, \cite[Proposition~1]{Browd62} and \cite[Proposition~6.3]{ShaoSim13}, one can verify that
\begin{align*}
(P,F) \in C^{\omega}(\mho, \mathcal{L}(E_1,E_0)\times E_0).
\end{align*}
For $\rho_0\in \mho$, now the surface diffusion flow equation \eqref{SDF equation 1} can be rewritten as:
\begin{equation}
\label{SDF equation 2}
\begin{cases}
\rho_t+P(\rho)\rho=F(\rho),\\
\rho(0)=\rho_0{,}
\end{cases}
\end{equation}
Given $(x,\xi)\in \Q\times S^{m-1}$ with $|\xi|=1$, estimate the principal symbol of $(P(\rho))_{\kappa}$:
\begin{align*}
&\hat{\sigma}(P(\rho))^{\pi}_{\kappa}(x,\xi)=\frac{1}{2}\sigma^{sr}(\rho(x))\xi_{s}\xi_{r}g_{\Gamma}^{ij}(x)\xi_{i}\xi_{j}\geq{c|\xi|^4}
\end{align*}
in every local chart $(\Uk,\varphi_{\kappa})$ for some $c>0$ by the compactness of $\M$. Thus $P(\rho)$ is normally elliptic for any $\rho\in \mho$. By \cite[Theorem~3.6]{ShaoSim13}, we infer that 
\begin{align*}
P(\rho)\in\mathcal{M}_{\frac{1}{2}}(E_1,E_0),\hspace{1em}\text{ for any }\rho\in \mho.
\end{align*}
\smallskip
Owing to \cite[Theorem~4.1]{CleSim01}, we can restate the result in \cite{EscMaySim98} as follows:
\begin{theorem}
\label{existence and uniqueness of SDF}
For any $\rho_0\in \mho:=\{\rho\in bc^{2+\alpha}(\M):\|\rho\|^{\M}_{\infty}< a \}$ for some $\alpha\in (0,1)$ and sufficiently small $a>0$, there exists a unique solution $\hat{\rho}$ to equation \eqref{SDF equation 2} with maximal interval of existence $J(\rho_0):=[0,T(\rho_0))$:
\begin{align*}
\hat{\rho}\in C^1_{\frac{1}{2}}(J(\rho_0),bc^{\alpha}(\M))\cap C_{\frac{1}{2}}(J(\rho_0),bc^{4+\alpha}(\M))\cap C(J(\rho_0),\mho).
\end{align*}
\end{theorem}
\goodbreak
We are now ready to prove the analyticity of the solution $\hat{\rho}$ as we did for the Ricci flow. Set $G(\rho):=P(\rho)\rho-F(\rho)$ for $\rho\in E_1\cap \mho$. Given $(t_0,\p)\in\dot{J}(\rho_0)\times{\M}$, we define ${\ttl}$ within ${\B}$ for sufficiently small $r$. Henceforth, we always use the notation $\hat{\rho}$ exclusively for the solution to \eqref{SDF equation 2}. Set $u:={\hat{\rho}}_{\lambda,\mu}$. Then $u$ satisfies the equation
\begin{align*}
u_t=-(1+\xi^{\prime}\lambda){\tu}G({\tui}u)+{B}_{\lambda,\mu}(u).
\end{align*}
Choose $J:=[\varepsilon,T]\subset\dot{J}(\rho_0)$ with $t_0\in \mathring{J}$. We put
\begin{align*}
{\ez}(J):=C(J,E_0),\hspace{1em} {\ef}(J):=C(J,E_1)\cap{C^1(J,E_0)},
\end{align*}
and
\begin{align*}
{\Ub}(J):=\{u\in{\ef}(J):\|u\|^{J\times\M}_{\infty}<a\}.
\end{align*}
Define a map $\Phi:{\Ub}(J)\times{\B}\rightarrow{\ez}(J)\times{E_1}$ as
\begin{align*}
\Phi(\rho,(\lambda,\mu))\mapsto\dbinom{\rho_t+(1+\xi^{\prime}\lambda){\tu}G({\tui}\rho)-{B}_{\lambda,\mu}(\rho)}{\gamma_{\varepsilon}(\rho)-\hat{\rho}(\varepsilon)}.
\end{align*}
Note that $\Phi(\hat{\rho}_{\lambda,\mu},(\lambda,\mu))=\dbinom{0}{0}$ for any $(\lambda,\mu)\in{\B}$.
\smallskip\\
(i) As discussed in Section 4, one obtains that
\begin{align*}
D_1\Phi(\hat{\rho},(0,0))\rho=\dbinom{v_t+DG(\hat{\rho})\rho}{\gamma_{\varepsilon}\rho}.
\end{align*}
The principal part of $DG(\hat{\rho})$ coincides with that of $P(\hat{\rho})$. Therefore, it follows from \cite[Theorem~3.6]{ShaoSim13} and an analogous argument to Lemma~\ref{section 7: MR of linearization} that
\begin{align*}
D_1\Phi(\hat{\rho},(0,0))\in {\Lis}({\Ub}(J)\times{\B},{\ez}(J)\times E_1).
\end{align*}
(ii) Adopting the decomposition in Section 4, we obtain that
\begin{align*}
G=\sum_{\kappa\in\mathfrak{K}}\tilde{\pi}_{\kappa}G. 
\end{align*}
By the proceeding discussion, the local expression of $\tilde{\pi}_{\kappa}G(\rho)$ in $({\Uk},\varphi_{\kappa})$ reads as
\begin{align*}
(\tilde{\pi}_{\kappa}G(\rho))_{\kappa}=\beta_{\kappa}^{2n}(\rho)\frac{{\kf}\tilde{\pi}_{\kappa}\mathcal{S}_{\kappa}(\rho,\cdots,\partial_{ijkl}\rho)}{\det[g_{\Gamma}]^{s_1} \det[\sigma(\rho)]^{s_2} \mathcal{B}_{\kappa}(\rho)}
\end{align*}
for every $\rho\in{\Ub}(J)$ and some $n,s_1,s_2\in\N$. Here $\mathcal{S}_{\kappa}$ is a polynomial of $\rho$ and its derivatives up to fourth order whose coefficients belong to $C^{\omega}(\Q)\cap BC^{\infty}(\Q)$. Meanwhile, $\mathcal{B}_{\kappa}(\rho)$ is a polynomial of $\rho$ with coefficients in $C^{\omega}(\Q)\cap BC^{\infty}(\Q)$. Observe that the components of the metric $g_{\Gamma}$ and $\sigma(\rho)$ are polynomials of $\rho$ and $\partial_j \rho$ with coefficients belonging to $C^{\omega}(\Q)\cap BC^{\infty}(\Q)$. Thus an analogous argument as for the Ricci flow applies to the scalar function $\tilde{\pi}_{\kappa}G(\rho)$. It implies that
\begin{align*}
\Phi\in{C^{\omega}({\Ub}(J)\times{\B},{\ez}(J)\times{E_1})}.
\end{align*}

Now the implicit function theorem and Theorem~\ref{S1: Theorem} yields the main theorems of this section:
\begin{theorem}
\label{S5: analyticity of SDF}
The solution $\hat{\rho}$ in Theorem~\ref{existence and uniqueness of SDF} satisfies $\hat{\rho}\in C^{\omega}(\dot{J}(\rho_0)\times\M)$.
\end{theorem}
\begin{cor}
\label{S5: unique continuation}
Suppose that $\rho$ and $\tilde{\rho}$ are two solutions to \eqref{SDF equation 2} with initial data $\rho_0,\tilde{\rho}_0\in\mho$ with $\mho$ defined as in Theorem~\ref{existence and uniqueness of SDF} for some $\alpha\in (0,1)$ on $J:=[0,T)$ belonging to the class $C^1_{\frac{1}{2}}(J,bc^{\alpha}(\M))\cap C_{\frac{1}{2}}(J,bc^{4+\alpha}(\M))$. If there exists some $t_0\in\dot{J}$ such that $\rho(t_0)=\tilde{\rho}(t_0)$, then $\rho(t)\equiv \tilde{\rho}(t)$ on $J$.
\end{cor}
\begin{theorem}
\label{S5: SDF: main theorem}
Suppose that $\Gamma_0$ is a compact closed embedded oriented hypersurface in ${\R}^{m+1}$ belonging to the class $C^{2+\alpha}$ for some $\alpha\in (0,1)$. Then the surface diffusion flow \eqref{SDF equation 1} has a local solution $\Gamma=\{\Gamma(t):t\in[0,T)\}$ for some $T>0$. Moreover,
\begin{align*}
\mathcal{M}:=\bigcup_{t\in(0,T)}(\{t\}\times\Gamma(t))
\end{align*}
is a real analytic hypersurface in ${\R}^{m+2}$. In particular, each manifold $\Gamma(t)$ is real analytic for $t\in(0,T)$. 
\end{theorem}
\begin{proof}
Note that $C^{2+\alpha}(\M)\hookrightarrow bc^{2+s}(\M)$ for any $s\in (0,\alpha)$. For each $(t_0,q)\in\mathcal{M}=\bigcup_{t\in\dot{J}(\rho_0)}(\{t\}\times\Gamma(t))$, there exists a $\p\in{\M}$ such that $\Psi_{\rho}(t_0,\p)=q$. Here $\Gamma(t)={\im}(\Psi_{\rho}(t,\cdot))$. Theorem \ref{S5: analyticity of SDF} states that there exists a local patch $(\Uk,\varphi_{\kappa})$ such that $\p\in\Uk$ and $\rho\circ\psi_{\kappa}$ is real analytic in $\dot{J}(\rho_0)\times{\Q}$. 
Therefore, we conclude 
\begin{align*}
[(t,x)\mapsto (t,\psi_{\kappa}(x)+\rho(t,\psi_{\kappa}(x))\nu_{\M}(\psi_{\kappa}(x))]\in{C^{\omega}(\dot{J}(\rho_0)\times{\Q},\mathcal{M})}.
\end{align*}
This proves the assertion.
\end{proof}
\medskip

\section{\bf The Mean Curvature Flow}

The averaged mean curvature flow problem, or sometimes been called volume preserving mean curvature flow problem, consists in looking for a family of smooth hypersurfaces $\Gamma=\{\Gamma(t):t\geq{0}\}$ in ${\R}^{m+1}$ satisfying the following equation
\begin{align}
\label{MCF equation 1}
\begin{cases}
V(t)=H_{\Gamma(t)}-h_{\Gamma(t)},\\
\Gamma(0)=\Gamma_0{,}
\end{cases}
\end{align}
where $V(t)$ and $H_{\Gamma(t)}$ have the same meaning as in the previous section, while $h_{\Gamma(t)}$ stands for the average of the mean curvature on $\Gamma(t)$, that is, for $t\geq{0}$
\begin{align*}
h_{\Gamma(t)}:=\frac{1}{\int\limits_{\Gamma(t)}\, d V_{g_{\Gamma}}}\int\limits_{\Gamma(t)}H_{\Gamma(t)}\, d V_{g_{\Gamma}}
\end{align*}
with ${g_{\Gamma}}$ being the Euclidean metric on $\Gamma(t)$. For $0<s<\alpha<1$ and $\gamma=\frac{2+s-\alpha}{2}$, let 
\begin{align*}
E_0:=bc^{\alpha}({\M}),\hspace{1em} E_1:=bc^{2+\alpha}({\M}),\hspace{1em}, E_{\gamma}:=(E_0,E_1)^0_{\gamma,\infty}=bc^{2+s}(\M), 
\end{align*}
and
\begin{align*}
W_{s}:=\{\rho\in E_{\gamma}: \|\rho\|_{\infty}^{\M}<a\}.
\end{align*}
If we start with an initial hypersurface $\Gamma_0$ in the class $bc^{2+s}$, then we can find a real analytic compact closed embedded hypersurface $({\M},g)$ with $g$ being the Euclidean metric on $\M$ and a function $\rho:[0,T)\times{\M}\rightarrow(-a,a)$ for some $T>0$ and sufficiently small $a>0$ such that $\{\Gamma(t):t\in [0,T)\}$ can be characterized as in \eqref{S4: Ref-mfd & para}. 

We identify $\rho(t,\psi_{\kappa}(\cdot))$ with $\rho(t,\cdot)$. Let $\beta(\rho)$, $L^{\M}$, $P_1(\rho)$ and $F_1(\rho)$ be the same as in Section 5. Following \cite[formula~(25)]{PruSim13}, we set 
\begin{align*}
D(\rho):=\frac{\alpha(\rho)}{ \beta(\rho)}, \hspace*{1em}\text{ with }\hspace*{.5em}\alpha(\rho):=\det(I-\rho K^{\M}),
\end{align*}
where $K^{\M}=[g]^{-1}L^{\M}$ is the shape matrix of $(\M,g)$. Note that $\alpha(\rho)$ is a polynomial of $\rho$ with $BC^{\infty}\cap C^{\omega}$-coefficients in every local chart. Define
\begin{align*}
&P(\rho)h:=-\frac{1}{\beta(\rho)}(P_1(\rho)h - \frac{1}{\int\limits_{\M}D(\rho)\, d V_g}\int\limits_{\M}P_1(\rho)h D(\rho)\, d V_g ),\hspace{1em}\rho\in W_{s},h\in{E_1},\\
&F(\rho):=\frac{1}{\beta(\rho)}(F_1(\rho) - \frac{1}{\int\limits_{\M}D(\rho)\, d V_g}\int\limits_{\M}F_1(\rho) D(\rho)\, d V_g ),\hspace{2.5em}\rho\in W_{s}.\\
\end{align*}
To simplify the notations, we set
\begin{itemize}
\item $\displaystyle B(\rho)h:=\frac{1}{\beta(\rho)\int\limits_{\M}D(\rho)\, d V_g}\int\limits_{\M}P_1(\rho)h D(\rho)\, d V_g$, \hspace{1em}for $\rho\in W_s$ and $h\in{E_1}$,
\item $\displaystyle A(\rho):=\frac{1}{\beta(\rho)\int\limits_{\M}D(\rho)\, d V_g}\int\limits_{\M}F_1(\rho) D(\rho)\, d V_g$, \hspace{1em}for $\rho\in W_s$,
\item $G(\rho):=P(\rho)\rho-F(\rho)$, and $K(\rho):=A(\rho)+B(\rho)\rho$, \hspace{1em}for $\rho\in E_1\cap W_s$.
\end{itemize}
By \cite[Proposition~1]{Browd62} and \cite[Proposition~6.3]{ShaoSim13}, one checks
\begin{align}
\label{S6: Q, beta, P1, F1}
(Q,\frac{1}{\beta}) \in{C}^{\omega}(W_{s},E_0\times E_0), \hspace{.5em}{\text{and }}\hspace{.5em} (P_1,F_1)\in{C}^{\omega}(W_{s},\mathcal{L}(E_1,E_0)\times E_0).
\end{align}
Following the proof in \cite{EscSim98}, we can show that $B(\rho)\in\mathcal{L}(BC^2({\M}),E_0)$ for any $\rho\in W_{s}$. 
On account to
\begin{align*}
[f\mapsto \int\limits_{\M} f\, d V_g]\in \L(BC(\M),\R),
\end{align*}
\eqref{S6: Q, beta, P1, F1} and \cite[Proposition~6.3]{ShaoSim13}, we have
\begin{align}
\label{S6: regularity of A,B}
B\in{C}^{\omega}(W_{s},\mathcal{L}(BC^2({\M}),E_0)), \hspace{.5em} \text{and likewise,} \hspace{.5em}A\in{C}^{\omega}(W_{s},E_0).
\end{align}
The above arguments enable us to translate equation \eqref{MCF equation 1} into
\begin{align}
\label{MCF equation 2}
\begin{cases}
\rho_t+P(\rho)\rho=F(\rho),\\
\rho(0)=\rho_0{,}
\end{cases}
\end{align}
where $\rho_0\in W_{s}$. In particular,
\begin{align*}
(P,F)\in{C}^{\omega}(W_{s},\mathcal{L}(E_1,E_0)\times E_0).
\end{align*}
By examining the symbol of the principal part for $(P(\rho))_{\kappa}$, that is,
\begin{align*}
\hat{\sigma}(P(\rho))^{\pi}_{\kappa}(x,\xi)=\frac{1}{2}g_{\Gamma}^{ij}(x)\xi_i\xi_j,\hspace{1em}(x,\xi)\in \Q\times S^{m-1},
\end{align*}
it is easily seen that $P(\rho)$ is normally elliptic for any $\rho\in W_{s}$. By \eqref{S6: regularity of A,B}, $B(\rho)$ is a lower order perturbation compared to $P(\rho)-B(\rho)$. Following the discussion in the previous sections and \cite[Lemma~2.7(c)]{CleSim01}, the result in \cite{EscSim98} can be restated as: 
\begin{theorem}
\label{S6: well-p AMCF}
Suppose that $\rho_0\in W_{s}:=\{u\in bc^{2+s}(\M):\|\rho\|^{\M}_{\infty}<a\}$ with some $s\in (0,1)$ and $a>0$ sufficiently small. Pick $s<\alpha<1$ and $\gamma=\frac{2+s-\alpha}{2}$. Then equation \eqref{MCF equation 2} has a unique solution $\hat{\rho}$ on the maximal interval of existence $J(\rho_0):=[0,T(\rho_0))$ with some $T(\rho_0)>0$:
\begin{align*}
\hat{\rho}\in C^1_{1-\gamma}(J(\rho_0),bc^{\alpha}(\M))\cap C_{1-\gamma}(J(\rho_0),bc^{2+\alpha}(\M))\cap C(J(\rho_0),W_{s}) .
\end{align*}
\end{theorem}
For any $(t_0,\p)\in\dot{J}(\rho_0)\times{\M}$, choose $J:=[\varepsilon,T]\subset\dot{J}(\rho_0)$ with $t_0\in \mathring{J}$. Define ${\ez}(J)$, ${\ef}(J)$ and ${\Ub}(J)$ as in Section 5. In the sequel, we always use $\hat{\rho}$ exclusively for the solution in Theorem~\ref{S6: well-p AMCF}. Define $\Phi:{\Ub}(J)\times{\B}\rightarrow{\ez}(J)\times{E_1}$ by
\begin{align*}
\Phi(\rho,(\lambda,\mu))\mapsto\dbinom{\rho_t+(1+\xi^{\prime}\lambda){\tu}G({\tui}\rho)-{B}_{\lambda,\mu}(\rho)}{\gamma_{\varepsilon}(\rho)-\hat{\rho}(\varepsilon)}.
\end{align*}
By a similar argument to Section 4, \eqref{S5: beta} and \cite[Proposition~6.4]{ShaoSim13}, we obtain
\begin{align}
\label{S6: regularity of transformed beta}
[(\rho,\mu)\mapsto\tu \beta({\tui}\rho)]\in{C^{\omega}({\Ub}(J)\times\B},\ez(J)).
\end{align}
\begin{prop}
\label{S6: regularity of transformed integral}
$[(\rho,\mu)\mapsto{\tu}K({\tui}\rho)]\in{C^{\omega}({\Ub}(J)\times{\B},\ez(J))}$.
\end{prop}
\begin{proof}
Pick $\rho\in {\Ub}(J)$. First notice that ${\tu}[\beta({\tui}\rho) K({\tui}\rho)]=\beta({\tui}\rho)K({\tui}\rho)$.
\smallskip\\
(i) For any $t\in J$, we decompose $I(\mu)(t):=\int\limits_{\M} F_1({\tui}(t)\rho(t))D({\tui}(t)\rho(t))\, d V_g$ into
\begin{align*}
\underbrace{\int\limits_{\psi_{\iota}(B_4)} \!\! F_1({\tui}(t)\rho(t))D({\tui}(t)\rho(t))\, d V_g}_{I_1(\mu)(t)}+ \underbrace{\int\limits_{\psi_{\iota}(B_4)^C} \!\! F_1({\tui}(t)\rho(t))D({\tui}(t)\rho(t))\, d V_g}_{I_2(\mu)(t)}.
\end{align*}
By Lemma~\ref{quoting lemma} and the independence of $I_2(\mu)(t)$ on $\mu$ for all $t\in J$, we infer that $[\mu\mapsto I_2(\mu)]\in C^{\omega}({\B},C(J))$. On the other hand, one computes
\begin{align*}
&\quad I_1(\mu)(t)=\int\limits_{B_4} F_{1,\iota}({\ttui}(t)\rho(t))D_{\iota}({\ttui}(t)\rho(t))\sqrt{\det[g]}\, dx \\
&=\int\limits_{B_4} {\ttu}(t)F_{1,\iota}({\ttui}(t)\rho(t)){\ttu}(t)D_{\iota}({\ttui}(t)\rho(t))  {\ttu(t)}\sqrt{\det[g]}\; |\det(D\theta_{\xi(t)\mu})| \, dy,
\end{align*}
where $F_{1,\iota}(\rho):={\kf}F_1(\rho)$ and $\beta_{\iota}(\rho)$, $D_{\iota}(\rho)$ are defined alike. Recall that notation-wise we do not distinguish $\rho$ from ${\kfk}\rho$. 

(ii) Let $\mathbb{E}:=C(J,BC(B_4))$. $|\det(D\theta_{\xi(t)\mu})|$ is a polynomial of $\mu_j$ with $BC^{\infty}$-coefficients multiplied by $\xi^n(t)$ with $n\in\N_0$, since $\theta_{\xi(t)\mu}\in{\Dfi}(\Qi)$. Accordingly, we infer that
\begin{align*}
[\mu\mapsto |\det(D\theta_{\xi(\cdot)\mu})|]\in C^{\omega}({\B},\mathbb{E}).
\end{align*}
(iii) Because $\det[g]\in BC^{\infty}(\Qi)\cap C^{\omega}(\Qi)$, \cite[Proposition~5.2(a)]{EscPruSim03} and \cite[Proposition~6.4, Remark~6.5]{ShaoSim13} imply that
\begin{align*}
[\mu\mapsto {\ttu}\sqrt{\det[g]}]\in C^{\omega}({\B},\mathbb{E}) .
\end{align*}
(iv) In $(\Uki,\varphi_{\iota})$, we can decompose $D_{\iota}^2(\rho)$ into
\begin{align*}
D_{\iota}^2(\rho)=\frac{1}{ \mathcal{P}_{\iota}(\rho)}(\mathcal{B}_{\iota}+ \sum_{k} \mathcal{A}_{\iota,k_1}\rho\cdots \mathcal{A}_{\iota,k_s}\rho) ,
\end{align*}
where $\mathcal{B}_{\iota}\in BC^{\infty}(\Qi) \cap C^{\omega}(\Qi)$, and $\mathcal{A}_{\iota,r}$ are linear differential operators of order at most one with coefficients
\begin{align*}
a^{\iota,r}_{\alpha}\in BC^{\infty}(\Qi) \cap C^{\omega}(\Qi).
\end{align*}
$\mathcal{P}_{\iota}(\rho)$ is a polynomial of $\rho$ with $BC^{\infty}(\Qi)\cap C^{\omega}(\Qi)$-coefficients. Argue similarly to Section 4 for functions on $\Qi$ and use \cite[Proposition~5.2]{EscPruSim03}, \cite[Proposition~6.4, Remark~6.5]{ShaoSim13}. We infer that
\begin{align*}
[(\rho,\mu)\mapsto {\ttu}D_{\iota}({\ttui}\rho)]\in C^{\omega}({\Ub}(J)\times{\B},\mathbb{E}). 
\end{align*}
(v) Taking into account the fact that $\beta_{\iota}(\rho) \leq 1$, \eqref{S5: beta} and \cite[Proposition~6.4, Remark~6.5]{ShaoSim13}, a semblable argument to (iv) implies that
\begin{align*}
[(\rho,\mu)\mapsto {\ttu}\beta_{\iota}({\ttui}\rho)]\in C^{\omega}({\Ub}(J)\times\B,\mathbb{E}).
\end{align*}
Since $\frac{F_{1,\iota}(\rho)}{\beta_{\iota}(\rho)}$ is a polynomial of $\rho$ and $\partial_j \rho$ with $BC^{\infty}\cap C^{\omega}$-coefficients, it is an immediate consequence of the point-wise multiplier theorem on $B_4$ that 
\begin{align*}
[(\rho,\mu)\mapsto {\ttu}F_{1,\iota}({\ttui}\rho)]\in C^{\omega}({\Ub}(J)\times\B,\mathbb{E}).
\end{align*}
(vi) Combining all the above results, the point-wise multiplier theorem now yields
\begin{align*}
[(\rho,\mu)\mapsto {\ttu}\{F_{1,\iota}({\ttui}\rho)D_{\iota}({\ttui}\rho)\sqrt{\det[g]}\} |\det(D\theta_{\xi\mu})|]\in{C^{\omega}({\Ub}(J)\times{\B},\mathbb{E})}. 
\end{align*}
It follows from 
\begin{align}
\label{S6: integral B_4}
[f\mapsto \int\limits_{B_4} f\, dx]\in \L(BC(B_4),\R),
\end{align}
that $[\mu\mapsto I_1(\mu)]\in C^{\omega}({\B}, C(J))$. Finally, we are in a position to conclude 
\begin{align*}
[(\rho,\mu)\mapsto  \int\limits_{\M} F_1({\tui}\rho)D({\tui}\rho)\, d V_g]\in{C^{\omega}({\Ub}(J)\times{\B},C(J))}.
\end{align*}
(vii) The decompositions in (i), (iv) and \eqref{S6: integral B_4} imply that 
\begin{align*}
[(\rho,\mu)\mapsto \frac{1}{\int\limits_{\M} D({\tui}\rho)\, d V_g}]\in C^{\omega}({\Ub}(J)\times{\B},C(J)).
\end{align*}
Modifying (i)-(vi) in an obvious way, we obtain that
\begin{align*}
[(\rho,\mu)\mapsto \int\limits_{\M} P_1({\tui}\rho){\tui}\rho D({\tui}\rho)\, d V_g] \in C^{\omega}({\Ub}(J)\times{\B},C(J)).
\end{align*}
Proposition~\ref{pointwise multiplication properties} and \eqref{S6: regularity of transformed beta} complete the proof.
\end{proof}
Following the proof of Section 4, we obtain
\begin{align}
\label{S6: regularity of H}
[(\rho,\mu)\mapsto {\tu}H_{{\tui}\rho}]\in C^{\omega}({\Ub}(J)\times{\B},{\ez}(J)).
\end{align}
We thus conclude from the above propositions that 
\begin{align*}
\Phi\in{C}^{\omega}({\Ub}(J)\times{\B},{\ez}(J)\times{E_1}).
\end{align*}
Similarly, it is immediate from the discussion in the previous sections that
\begin{align*}
D_1\Phi(\hat{\rho},(0,0))\in{\Lis}({\ef}(J),{\ez}(J)\times{E_1}).
\end{align*}
The implicit function theorem and Theorem \ref{S1: Theorem} thus yield the following theorems:
\begin{theorem}
Suppose that $\Gamma_0$ is a compact closed embedded oriented hypersurface in ${\R}^{m+1}$ belonging to the class $C^{2+s}$ for some $s\in (0,1)$. Then the averaged mean curvature flow \eqref{MCF equation 1} has a local solution $\Gamma=\{\Gamma(t):t\in[0,T)\}$ for some $T>0$. Moreover,
\begin{center}
$\mathcal{M}:=\bigcup_{t\in(0,T)}(\{t\}\times\Gamma(t))$ is a real analytic hypersurface in ${\R}^{m+2}$.
\end{center}
\end{theorem}
\begin{remark}
The case that $\Gamma_0$ belongs to the class $C^2$ is permissible in the above theorem, provided that we consider $P,F$ as functions defined on an open subset $\{\rho\in bc^{1+s}(\M):\|\rho\|^{\M}_{\infty}<a\}$ of $bc^{1+s}(\M)$ in the proof of Theorem~\ref{S6: well-p AMCF} and use the quasi-linear structure of \eqref{MCF equation 2}. 
\end{remark}

Following an analogous discussion, we can show that the solution to the mean curvature flow 
\begin{align}
\label{RMCF equation 1}
\begin{cases}
V(t)=H_{\Gamma(t)},\\
\Gamma(0)=\Gamma_0{.}
\end{cases}
\end{align}
immediately becomes analytic.
\begin{theorem}
Suppose that $\Gamma_0$ is a compact closed embedded oriented hypersurface in ${\R}^{m+1}$ belonging to the class $C^{2+s}$ for some $0<s<1$.  Then the mean curvature flow \eqref{RMCF equation 1} has a local solution $\Gamma=\{\Gamma(t):t\in[0,T)\}$ for some $T>0$. Moreover,
\begin{center}
$\mathcal{M}:=\bigcup_{t\in(0,T)}(\{t\}\times\Gamma(t))$ is a real analytic hypersurface in ${\R}^{m+2}$.
\end{center}
\end{theorem}
\begin{remark}
An analogous result to Corollary~\ref{S5: unique continuation} holds for the (averaged) mean curvature flow.
\end{remark}

\section*{Acknowledgements}
I would like to express my great gratitude to my advisor, Gieri Simonett, for his help and invaluable suggestions throughout my research.

\end{document}